\documentclass[a4paper,oneside,10pt]{article}%
\usepackage{amsmath}
\usepackage{amsfonts}
\usepackage{amssymb}
\usepackage{graphicx}
\usepackage{color}
\usepackage[square,numbers,sort&compress]{natbib}%
\providecommand{\U}[1]{\protect \rule{.1in}{.1in}}

\newtheorem{theorem}{Theorem}[section]

\newtheorem{definition}[theorem]{Definition}

\newtheorem{proposition}[theorem]{Proposition}
\newtheorem{remark}[theorem]{Remark}

\newenvironment{proof}[1][Proof]{\noindent \textbf{#1.} }{\  $\Box$}
\numberwithin{equation}{section}

\begin{document}

\title{Forward-backward stochastic differential equations driven by $G$-Brownian
motion under weakly coupling condition}
\author{Xiaojuan Li\thanks{Zhongtai Securities Institute for Financial Studies,
Shandong University, Jinan 250100, China, lxj110055@126.com. Research
supported by Natural Science Foundation of Shandong Province (No.
ZR2014AP005). } }
\maketitle

\textbf{Abstract}. In this paper, we obtain the existence and uniqueness
theorem of $L^{p}$-solution for coupled forward-backward stochastic
differential equations driven by $G$-Brownian motion ($G$-FBSDEs) with
arbitrary $T$ under weakly coupling condition. Specially, the result for
$p\in(1,2)$ is completely different from the one for $p\geq2$. Furthermore, by
considering the dual linear FBSDE under a suitable reference probability, we
establish the comparison theorem for $G$-FBSDEs under weakly coupling condition.

{\textbf{Key words}. } $G$-expectation; $G$-Brownian motion; Backward
stochastic differential equation; Comparison theorem

\textbf{AMS subject classifications.} 60H10

\addcontentsline{toc}{section}{\hspace*{1.8em}Abstract}

\section{Introduction}

The classical fully coupled forward-backward stochastic differential equation
(FBSDE) has the following form%
\begin{equation}
\left \{
\begin{array}
[c]{rl}%
dX_{t}= & b(t,X_{t},Y_{t},Z_{t})dt+\sigma(t,X_{t},Y_{t},Z_{t})dW_{t},\\
dY_{t}= & f(t,X_{t},Y_{t},Z_{t})dt+Z_{t}dW_{t},\\
X_{0}= & x_{0},\text{ }Y_{T}=\phi(X_{T}),
\end{array}
\right.  \label{e-0-1}%
\end{equation}
where $W$ is classical standard Brownian motion. There are many literatures to
study the existence and uniqueness of the solution to FBSDE (\ref{e-0-1}).
Antonelli \cite{An} first obtained the existence and uniqueness result by
fixed point approach for small $T$. Ma et al. \cite{MPY} introduced the four
step scheme to first obtain the existence and uniqueness theorem for arbitrary
$T$. Hu, Peng \cite{HuP} and Yong \cite{Yong} introduced the method of
continuation to study FBSDE (\ref{e-0-1}). Pardoux and Tang \cite{P-T}
obtained the existence and uniqueness theorem for arbitrary $T$ by fixed point
approach under weakly coupling condition. For more results on this topic, the
reader may refer to \cite{De, MWZZ, PW} and the references therein. The
applications of the theory of FBSDEs in finance can be found in Ma and Yong's
book \cite{MY}. Wu \cite{Wu} studied the comparison theorem for FBSDE
(\ref{e-0-1}) by duality method (see also \cite{HJX, HJX1}).

Motivated by volatility uncertainty in finance (see \cite{ALP, Ly}), Peng
\cite{P07a, P08a} introduced a type of consistent sublinear expectation,
called the $G$-expectation $\mathbb{\hat{E}}[\cdot]$. The related $G$-Brownian
motion $B$ and It\^{o}'s calculus with respect to $B$ were constructed.
Moreover, the theory of stochastic differential equation driven by
$G$-Brownian motion ($G$-SDE) has been established.

Hu et al. \cite{HJPS1} studied the backward stochastic differential equation
driven by $G$-Brownian motion ($G$-BSDE). The theory of quadratic $G$-BSDE has
been established in \cite{HYH}, and the wellposedness of a type of
multi-dimensional $G$-BSDE can be found in \cite{Liu}. Soner et al.
\cite{STZ11} (see also \cite{CST}) studied a new type of fully nonlinear BSDE,
called $2$BSDE, by different formulation and method. The theory of $2$BSDE
with random terminal time has been obtained in \cite{LRT}.

Recently, Lu and Song \cite{LS}, and Zheng \cite{Zh} studied the following
coupled forward-backward stochastic differential equation driven by
$G$-Brownian motion ($G$-FBSDE):%
\begin{equation}
\left \{
\begin{array}
[c]{rl}%
dX_{t}= & b(t,X_{t},Y_{t})dt+h(t,X_{t},Y_{t})d\langle B\rangle_{t}%
+\sigma(t,X_{t},Y_{t})dB_{t},\\
dY_{t}= & f(t,X_{t},Y_{t},Z_{t})dt+g(t,X_{t},Y_{t},Z_{t})d\langle B\rangle
_{t}+Z_{t}dB_{t}+dK_{t},\\
X_{0}= & x_{0}\in \mathbb{R}^{n},\text{ }Y_{T}=\phi(X_{T})\in \mathbb{R}.
\end{array}
\right.  \label{e-0-2}%
\end{equation}
By fixed point approach, they obtained that $G$-FBSDE (\ref{e-0-2}) has a
unique $L^{2}$-solution $(X,Y,Z,K)$ for small $T$. Wang and Yuan \cite{WY}
studied the minimal solution of $G$-FBSDE (\ref{e-0-2}) with monotone
coefficients under the assumption that $\sigma(\cdot)$ is independent of $Y$
and $n=1$.

In this paper, we first study the $L^{p}$-solution of $G$-FBSDE (\ref{e-0-2})
for arbitrary $T$ under weakly coupling condition. By fixed point approach, we
obtain that $G$-FBSDE (\ref{e-0-2}) has a unique $L^{p}$-solution $(X,Y,Z,K)$
with $p\geq2$ for arbitrary $T$ under weakly coupling condition. But for
$p\in(1,2)$, in order to get contractive mapping for $\hat{X}$, we need the
assumption that $\sigma(\cdot)$ does not depend on $Y$. The key reason is that
the Doob inequality for $G$-martingale (see \cite{STZ, Song11}) is different
from the classical case and
\[
\left(  \int_{0}^{T}|\hat{Y}_{t}|^{2}dt\right)  ^{p/2}\leq C\int_{0}^{T}%
|\hat{Y}_{t}|^{p}dt
\]
does not hold for $p\in(1,2)$.

It is well known that the comparison theorem plays an important role in the
theory of BSDEs. So, the other purpose of this paper is to establish the
comparison theorem for $G$-FBSDEs under weakly coupling condition. The key
point to prove the comparison theorem is to slove the linear $G$-FBSDE. Since
the solvability of the dual linear $G$-FBSDE is unknown, we cannot use the
method in \cite{HJPS} to prove the comparison theorem. In order to overcome
this difficulty, we must choose a suitable reference probability $P^{\ast}$
and consider the dual linear FBSDE under $P^{\ast}$. The BSDE in this dual
equation is different from the one in (\ref{e-0-1}) and studied in \cite{EH}.
By fixed point approach under weakly coupling condition, we can still obtain
the solvability of this dual linear FBSDE under $P^{\ast}$. Based on this, we
can further obtain the comparison theorem.

The paper is organized as follows. In Section 2, we recall some basic results
of $G$-expectations, $G$-SDEs and $G$-BSDEs. The existence and uniqueness
theorem, and the related estimates of $L^{p}$-solution for $G$-FBSDEs have
been established in Section 3. In Section 4, we obtain the comparison theorem
for $G$-FBSDEs. \ 

\section{Preliminaries}

We recall some basic results of $G$-expectations, $G$-SDEs and $G$-BSDEs. The
readers may refer to Peng's book \cite{P2019}, \cite{HJPS1} and \cite{HJPS}
for more details.

Let $T>0$ be given and let $\Omega_{T}=C_{0}([0,T];\mathbb{R}^{d})$ be the
space of $\mathbb{R}^{d}$-valued continuous functions on $[0,T]$ with
$\omega_{0}=0$. The canonical process $B_{t}(\omega):=\omega_{t}$, for
$\omega \in \Omega_{T}$ and $t\in \lbrack0,T]$. For any fixed $t\leq T$, set%
\[
Lip(\Omega_{t}):=\{ \varphi(B_{t_{1}},B_{t_{2}}-B_{t_{1}},\ldots,B_{t_{N}%
}-B_{t_{N-1}}):N\geq1,t_{1}<\cdots<t_{N}\leq t,\varphi \in C_{b.Lip}%
(\mathbb{R}^{d\times N})\},
\]
where $C_{b.Lip}(\mathbb{R}^{d\times N})$ denotes the space of bounded
Lipschitz functions on $\mathbb{R}^{d\times N}$.

Let $G:\mathbb{S}_{d}\rightarrow \mathbb{R}$ be a given monotonic and sublinear
function, where $\mathbb{S}_{d}$ denotes the set of $d\times d$ symmetric
matrices. In this paper, we only consider non-degenerate $G$, i.e., there
exists a $\gamma>0$ such that
\[
G(A)-G(B)\geq \frac{\gamma}{2}\mathrm{tr}[A-B]\text{ for }A\geq B\text{.}%
\]

Peng \cite{P07a, P08a} constructed a consistent sublinear expectation space
$(\Omega_{T},Lip(\Omega_{T}),\mathbb{\hat{E}},(\mathbb{\hat{E}}_{t}%
)_{t\in \lbrack0,T]})$, called $G$-expectation space, such that, for $0\leq
t<s\leq T$, $\xi_{i}\in Lip(\Omega_{t})$, $i\leq m$, $\varphi \in
C_{b.Lip}(\mathbb{R}^{m+d})$,
\[
\mathbb{\hat{E}}_{t}\left[  \varphi(\xi_{1},\ldots,\xi_{m},B_{s}%
-B_{t})\right]  =\psi(\xi_{1},\ldots,\xi_{m}),
\]
where $\psi(x_{1},\ldots,x_{m})=u(s-t,0)$, $u$ is the solution of the
following $G$-heat equation:%
\[
\partial_{t}u-G(D_{x}^{2}u)=0,\ u(0,x)=\varphi(x_{1},\ldots,x_{m},x).
\]
The canonical process $(B_{t})_{t\in \lbrack0,T]}$ is called the $G$-Brownian
motion under $\mathbb{\hat{E}}$.

For each $t\in \lbrack0,T]$, denote by $L_{G}^{p}(\Omega_{t})$ the completion
of $Lip(\Omega_{t})$ under the norm $||X||_{L_{G}^{p}}:=(\mathbb{\hat{E}%
}[|X|^{p}])^{1/p}$ for $p\geq1$. It is clear that $\mathbb{\hat{E}}_{t}$ can
be continuously extended to $L_{G}^{1}(\Omega_{T})$ under the norm
$||\cdot||_{L_{G}^{1}}$.

\begin{definition}
A process $(M_{t})_{t\leq T}$ is called a $G$-martingale if $M_{T}\in
L_{G}^{1}(\Omega_{T})$ and $\mathbb{\hat{E}}_{t}[M_{T}]=M_{t}$ for $t\leq T$.
\end{definition}

The following theorem is the representation theorem of $G$-expectation.

\begin{theorem}
(\cite{DHP11, HP09}) \label{th-2-1}There exists a unique weakly compact and
convex set of probability measures $\mathcal{P}$ on $(\Omega_{T}%
,\mathcal{B}(\Omega_{T}))$ such that%
\[
\mathbb{\hat{E}}[X]=\sup_{P\in \mathcal{P}}E_{P}[X]\text{ for all }X\in
L_{G}^{1}(\Omega_{T}),
\]
where $\mathcal{B}(\Omega_{T})=\sigma(B_{s}:s\leq T)$.
\end{theorem}

The capacity associated to $\mathcal{P}$ is defined by
\[
c(A):=\sup_{P\in \mathcal{P}}P(A)\text{ for }A\in \mathcal{B}(\Omega_{T}).
\]
A set $A\in \mathcal{B}(\Omega_{T})$ is polar if $c(A)=0$. A property holds
\textquotedblleft quasi-surely" (q.s. for short) if it holds outside a polar
set. In the following, we do not distinguish two random variables $X$ and $Y$
if $X=Y$ q.s.

In order to study $G$-FBSDE, we need the following spaces and norms.

\begin{itemize}
\item $M^{0}(0,T):=\left \{  \eta_{t}=\sum_{i=0}^{N-1}\xi_{i}I_{[t_{i}%
,t_{i+1})}(t):N\in \mathbb{N}\text{, }0=t_{0}<\cdots<t_{N}=T,\text{ }\xi_{i}\in
Lip(\Omega_{t_{i}})\right \}  $;

\item $||\eta||_{M_{G}^{\bar{p},p}(0,T)}:=\left(  \mathbb{\hat{E}}\left[
\left(  \int_{0}^{T}|\eta_{t}|^{\bar{p}}dt\right)  ^{p/\bar{p}}\right]
\right)  ^{1/p}$ for $\bar{p}$, $p>0$;

\item $M_{G}^{\bar{p},p}(0,T):=\left \{  \text{the completion of }%
M^{0}(0,T)\text{ under the norm }||\cdot||_{M_{G}^{\bar{p},p}(0,T)}\right \}  $
for $\bar{p}$, $p\geq1$;

\item $S^{0}(0,T):=\left \{  h(t,B_{t_{1}\wedge t},\ldots,B_{t_{N}\wedge
t}):N\in \mathbb{N}\text{, }0<t_{1}<\cdots<t_{N}=T,\text{ }h\in C_{b.Lip}%
(\mathbb{R}^{1+dN})\right \}  $;

\item $||\eta||_{S_{G}^{p}(0,T)}:=\left(  \mathbb{\hat{E}}\left[  \sup_{t\leq
T}|\eta_{t}|^{p}\right]  \right)  ^{1/p}$ for $p>0$;

\item $S_{G}^{p}(0,T):=\left \{  \text{the completion of }S^{0}(0,T)\text{
under the norm }||\cdot||_{S_{G}^{p}(0,T)}\right \}  $ for $p\geq1$.
\end{itemize}

For each $\eta^{i}\in M_{G}^{2,p}(0,T)$ with $p\geq1$, $i=1,\ldots,d$, denote
$\eta=(\eta^{1},\ldots,\eta^{d})^{T}\in M_{G}^{2,p}(0,T;\mathbb{R}^{d})$, the
$G$-It\^{o} integral $\int_{0}^{T}\eta_{t}^{T}dB_{t}$ is well defined. Similar
for $L_{G}^{p}(\Omega_{t};\mathbb{R}^{n})$ and $S_{G}^{p}(0,T;\mathbb{R}^{n})$.

For simplicity of presentation, we suppose $d=1$ throughout the paper. The
results still hold for $d>1$. Under this case, the non-degenerate $G$ is%
\[
G(a)=\frac{1}{2}(\bar{\sigma}^{2}a^{+}-\underline{\sigma}^{2}a^{-})\text{ for
}a\in \mathbb{R},
\]
where $0<\underline{\sigma}\leq \bar{\sigma}<\infty$. If $\underline{\sigma
}=\bar{\sigma}$, then $\bar{\sigma}^{-1}B$ is a classical standard Brownian
motion. So we suppose $\underline{\sigma}<\bar{\sigma}$ in the following.

Let $\langle B\rangle$ be the quadratic variation process of $B$. By Corollary
3.5.5 in Peng \cite{P2019}, we have%
\begin{equation}
\underline{\sigma}^{2}s\leq \langle B\rangle_{t+s}-\langle B\rangle_{t}\leq
\bar{\sigma}^{2}s\text{ for each }t\text{, }s\geq0. \label{e-2-1}%
\end{equation}
Since $B$ is a martingale under each $P\in \mathcal{P}$, by Theorem
\ref{th-2-1} and the Burkholder-Davis-Gundy inequality, for each $p>0$ and
$||\eta||_{M_{G}^{2,p}(0,T)}<\infty$, there exists a constant $C(p)>0$ such
that%
\begin{equation}
\mathbb{\hat{E}}\left[  \sup_{t\leq T}\left \vert \int_{0}^{t}\eta_{s}%
dB_{s}\right \vert ^{p}\right]  \leq C(p)\mathbb{\hat{E}}\left[  \left(
\int_{0}^{T}|\eta_{s}|^{2}d\langle B\rangle_{s}\right)  ^{p/2}\right]
\leq \bar{\sigma}^{p}C(p)\mathbb{\hat{E}}\left[  \left(  \int_{0}^{T}|\eta
_{s}|^{2}ds\right)  ^{p/2}\right]  . \label{e-2-2}%
\end{equation}

In the following, we consider the following $G$-FBSDE:%
\begin{equation}
\left \{
\begin{array}
[c]{rl}%
dX_{t}= & b(t,X_{t},Y_{t})dt+h(t,X_{t},Y_{t})d\langle B\rangle_{t}%
+\sigma(t,X_{t},Y_{t})dB_{t},\\
dY_{t}= & f(t,X_{t},Y_{t},Z_{t})dt+g(t,X_{t},Y_{t},Z_{t})d\langle B\rangle
_{t}+Z_{t}dB_{t}+dK_{t},\\
X_{0}= & x_{0}\in \mathbb{R}^{n},\text{ }Y_{T}=\phi(X_{T}),
\end{array}
\right.  \label{e-2-3}%
\end{equation}
where $b$, $h$, $\sigma:[0,T]\times \Omega_{T}\times \mathbb{R}^{n}%
\times \mathbb{R}\rightarrow \mathbb{R}^{n}$, $f$, $g:[0,T]\times \Omega
_{T}\times \mathbb{R}^{n}\times \mathbb{R}\times \mathbb{R}\rightarrow \mathbb{R}%
$, $\phi:\Omega_{T}\times \mathbb{R}^{n}\rightarrow \mathbb{R}$. We need the
following assumptions:

\begin{description}
\item[(H1)] There exists a $\beta>1$ such that $b(\cdot,x,y)$, $h(\cdot
,x,y)\in M_{G}^{1,\beta}(0,T;\mathbb{R}^{n})$, $\sigma(\cdot,x,y)\in
M_{G}^{2,\beta}(0,T;\mathbb{R}^{n})$, $f(\cdot,x,y,z)$, $g(\cdot,x,y,z)\in
M_{G}^{1,\beta}(0,T)$ and $\phi(x)\in L_{G}^{\beta}(\Omega_{T})$ for each
$(x,y,z)\in \mathbb{R}^{n+2}$;

\item[(H2)] There exist constants $L_{i}>0$, $i=1$, $2$, $3$, such that, for
each $t\leq T$, $\omega \in \Omega_{T}$, $x$, $x^{\prime}\in \mathbb{R}^{n}$,
$y$, $y^{\prime}$, $z$, $z^{\prime}\in \mathbb{R}$,%
\[%
\begin{array}
[c]{l}%
|b_{j}(t,x,y)-b_{j}(t,x^{\prime},y^{\prime})|+|h_{j}(t,x,y)-h_{j}(t,x^{\prime
},y^{\prime})|+|\sigma_{j}(t,x,y)-\sigma_{j}(t,x^{\prime},y^{\prime})|\\
\leq L_{1}|x-x^{\prime}|+L_{2}|y-y^{\prime}|,\text{ for }j=1,\ldots,n,\\
|f(t,x,y,z)-f(t,x^{\prime},y^{\prime},z^{\prime})|+|g(t,x,y,z)-g(t,x^{\prime
},y^{\prime},z^{\prime})|\\
\leq L_{3}|x-x^{\prime}|+L_{1}(|y-y^{\prime}|+|z-z^{\prime}|),\\
|\phi(x)-\phi(x^{\prime})|\leq L_{3}|x-x^{\prime}|,
\end{array}
\]
where $b(\cdot)=(b_{1}(\cdot),\ldots,b_{n}(\cdot))^{T}$, $h(\cdot
)=(h_{1}(\cdot),\ldots,h_{n}(\cdot))^{T}$, $\sigma(\cdot)=(\sigma_{1}%
(\cdot),\ldots,\sigma_{n}(\cdot))^{T}$.
\end{description}

Now we give the $L^{p}$-solution of $G$-FBSDE (\ref{e-2-3}), similar for
$G$-SDE and $G$-BSDE.

\begin{definition}
For each fixed $p\in(1,\beta)$, $(X,Y,Z,K)$ is called an $L^{p}$-solution of
$G$-FBSDE (\ref{e-2-3}) if the following properties hold:

\begin{description}
\item[(i)] $X\in S_{G}^{p}(0,T;\mathbb{R}^{n})$, $Y\in S_{G}^{p}(0,T)$, $Z\in
M_{G}^{2,p}(0,T)$, $K$ is a non-increasing $G$-martingale with $K_{0}=0$ and
$K_{T}\in L_{G}^{p}(\Omega_{T})$;

\item[(ii)] $(X,Y,Z,K)$ satisfies $G$-FBSDE (\ref{e-2-3}).
\end{description}
\end{definition}

The following is the standard estimates of $G$-SDE and $G$-BSDE.

\begin{theorem}
\label{th-2-2}Suppose assumptions (H1) and (H2) hold. For each $p\in(1,\beta)$
and $(y_{t}^{(i)})_{t\leq T}\in S_{G}^{p}(0,T)$, $i=1$, $2$. Let $(X_{t}%
^{(i)})_{t\leq T}\in S_{G}^{p}(0,T;\mathbb{R}^{n})$ be the solution of $G$-SDE%
\[
dX_{t}^{(i)}=b(t,X_{t}^{(i)},y_{t}^{(i)})dt+h(t,X_{t}^{(i)},y_{t}%
^{(i)})d\langle B\rangle_{t}+\sigma(t,X_{t}^{(i)},y_{t}^{(i)})dB_{t},\text{
}X_{0}^{(i)}=x_{0},
\]
for $i=1$, $2$. Then there exists a deterministic function $C_{1}%
(p,T,L_{1},\bar{\sigma})>0$, which is continuous in $p$, such that%
\begin{equation}
\mathbb{\hat{E}}\left[  \sup_{t\leq T}\left \vert X_{t}^{(1)}-X_{t}%
^{(2)}\right \vert ^{p}\right]  \leq C_{1}(p,T,L_{1},\bar{\sigma}%
)\mathbb{\hat{E}}\left[  \left(  \int_{0}^{T}(|\hat{b}_{t}|+|\hat{h}%
_{t}|)dt\right)  ^{p}+\left(  \int_{0}^{T}|\hat{\sigma}_{t}|^{2}dt\right)
^{p/2}\right]  , \label{e-2-4}%
\end{equation}
where $\hat{b}_{t}=b(t,X_{t}^{(2)},y_{t}^{(1)})-b(t,X_{t}^{(2)},y_{t}^{(2)})$,
$\hat{h}_{t}=h(t,X_{t}^{(2)},y_{t}^{(1)})-h(t,X_{t}^{(2)},y_{t}^{(2)})$,
$\hat{\sigma}_{t}=\sigma(t,X_{t}^{(2)},y_{t}^{(1)})-\sigma(t,X_{t}^{(2)}%
,y_{t}^{(2)})$.
\end{theorem}

\begin{proof}
For the convenience of the reader, we sketch the proof. Set $\hat{X}_{t}%
=X_{t}^{(1)}-X_{t}^{(2)}$. For each given $t_{0}\in \lbrack0,T]$ and $\delta
>0$, we have%
\[
\hat{X}_{t}=\hat{X}_{t_{0}}+\int_{t_{0}}^{t}\tilde{b}(s)ds+\int_{t_{0}}%
^{t}\tilde{h}(s)d\langle B\rangle_{s}+\int_{t_{0}}^{t}\tilde{\sigma}%
(s)dB_{s},\text{ }t\in \lbrack t_{0},t_{0}+\delta],
\]
where $|\tilde{b}(s)|=|b(s,X_{s}^{(1)},y_{s}^{(1)})-b(s,X_{s}^{(2)}%
,y_{s}^{(2)})|\leq nL_{1}|\hat{X}_{s}|+|\hat{b}_{s}|$, similarly, $|\tilde
{h}(s)|\leq nL_{1}|\hat{X}_{s}|+|\hat{h}_{s}|$, $|\tilde{\sigma}(s)|\leq
nL_{1}|\hat{X}_{s}|+|\hat{\sigma}_{s}|$. Then we get%
\[
\sup_{t\in \lbrack t_{0},t_{0}+\delta]}\left \vert \hat{X}_{t}\right \vert
^{p}\leq4^{p-1}\left \{  |\hat{X}_{t_{0}}|^{p}+\left(  \int_{t_{0}}%
^{t_{0}+\delta}|\tilde{b}(s)|ds\right)  ^{p}+\bar{\sigma}^{2p}\left(
\int_{t_{0}}^{t_{0}+\delta}|\tilde{h}(s)|ds\right)  ^{p}+\sup_{t\in \lbrack
t_{0},t_{0}+\delta]}\left \vert \int_{t_{0}}^{t}\tilde{\sigma}(s)dB_{s}%
\right \vert ^{p}\right \}  .
\]
By (\ref{e-2-2}), we can deduce%
\[
\mathbb{\hat{E}}\left[  \sup_{t\in \lbrack t_{0},t_{0}+\delta]}\left \vert
\int_{t_{0}}^{t}\tilde{\sigma}(s)dB_{s}\right \vert ^{p}\right]  \leq n^{p}%
\bar{\sigma}^{p}C(p)\mathbb{\hat{E}}\left[  \left(  \int_{t_{0}}^{t_{0}%
+\delta}|\tilde{\sigma}(s)|^{2}ds\right)  ^{p/2}\right]  .
\]
It is easy to verify that%
\begin{align*}
\left(  \int_{t_{0}}^{t_{0}+\delta}|\tilde{b}(s)|ds\right)  ^{p}  &
\leq2^{p-1}\left[  \left(  nL_{1}\int_{t_{0}}^{t_{0}+\delta}|\hat{X}%
_{s}|ds\right)  ^{p}+\left(  \int_{t_{0}}^{t_{0}+\delta}|\hat{b}%
_{s}|ds\right)  ^{p}\right] \\
&  \leq2^{p-1}(nL_{1}\delta)^{p}\sup_{t\in \lbrack t_{0},t_{0}+\delta
]}\left \vert \hat{X}_{t}\right \vert ^{p}+2^{p-1}\left(  \int_{t_{0}}%
^{t_{0}+\delta}|\hat{b}_{s}|ds\right)  ^{p}%
\end{align*}
and%
\begin{align*}
\left(  \int_{t_{0}}^{t_{0}+\delta}|\tilde{\sigma}(s)|^{2}ds\right)  ^{p/2}
&  \leq2^{p/2}\left[  \left(  2n^{2}L_{1}^{2}\int_{t_{0}}^{t_{0}+\delta}%
|\hat{X}_{s}|^{2}ds\right)  ^{p/2}+\left(  2\int_{t_{0}}^{t_{0}+\delta}%
|\hat{\sigma}_{s}|^{2}ds\right)  ^{p/2}\right] \\
&  \leq2^{p}(nL_{1})^{p}\delta^{p/2}\sup_{t\in \lbrack t_{0},t_{0}+\delta
]}\left \vert \hat{X}_{t}\right \vert ^{p}+2^{p}\left(  \int_{t_{0}}%
^{t_{0}+\delta}|\hat{\sigma}_{s}|^{2}ds\right)  ^{p/2}.
\end{align*}
Thus we obtain%
\begin{align*}
\mathbb{\hat{E}}\left[  \sup_{t\in \lbrack t_{0},t_{0}+\delta]}\left \vert
\hat{X}_{t}\right \vert ^{p}\right]   &  \leq4^{p-1}\mathbb{\hat{E}}\left[
|\hat{X}_{t_{0}}|^{p}\right]  +\lambda_{1}(\delta)\mathbb{\hat{E}}\left[
\sup_{t\in \lbrack t_{0},t_{0}+\delta]}\left \vert \hat{X}_{t}\right \vert
^{p}\right] \\
&  \  \  \ +\lambda_{2}\mathbb{\hat{E}}\left[  \left(  \int_{0}^{T}(|\hat{b}%
_{t}|+|\hat{h}_{t}|)dt\right)  ^{p}+\left(  \int_{0}^{T}|\hat{\sigma}_{t}%
|^{2}dt\right)  ^{p/2}\right]  ,
\end{align*}
where%
\[
\lambda_{1}(\delta)=8^{p-1}\left[  (1+\bar{\sigma}^{2p})(nL_{1}\delta
)^{p}+2C(p)(L_{1}n^{2}\bar{\sigma})^{p}\delta^{p/2}\right]  ,\text{ }%
\lambda_{2}=8^{p-1}\left[  1+\bar{\sigma}^{2p}+2C(p)(n\bar{\sigma}%
)^{p}\right]  .
\]
Choosing $\delta_{0}>0$ such that $\lambda_{1}(\delta_{0})=0.75$, then, for
$\delta \leq \delta_{0}\wedge(T-t_{0})$, we get%
\[
\mathbb{\hat{E}}\left[  \sup_{t\in \lbrack t_{0},t_{0}+\delta]}\left \vert
\hat{X}_{t}\right \vert ^{p}\right]  \leq4^{p}\mathbb{\hat{E}}\left[  |\hat
{X}_{t_{0}}|^{p}\right]  +4\lambda_{2}\mathbb{\hat{E}}\left[  \left(  \int
_{0}^{T}(|\hat{b}_{t}|+|\hat{h}_{t}|)dt\right)  ^{p}+\left(  \int_{0}^{T}%
|\hat{\sigma}_{t}|^{2}dt\right)  ^{p/2}\right]  .
\]
Thus we can deduce%
\[
\mathbb{\hat{E}}\left[  \sup_{t\leq T}\left \vert X_{t}^{(1)}-X_{t}%
^{(2)}\right \vert ^{p}\right]  \leq C_{1}(p,T,L_{1},\bar{\sigma}%
)\mathbb{\hat{E}}\left[  \left(  \int_{0}^{T}(|\hat{b}_{t}|+|\hat{h}%
_{t}|)dt\right)  ^{p}+\left(  \int_{0}^{T}|\hat{\sigma}_{t}|^{2}dt\right)
^{p/2}\right]  ,
\]
where%
\begin{equation}
C_{1}(p,T,L_{1},\bar{\sigma})=\frac{4\lambda_{2}}{4^{p}-1}\left(
\frac{4^{p(T+2\delta_{0})/\delta_{0}}-4^{p}}{4^{p}-1}-\frac{T}{\delta_{0}%
}\right)  . \label{e-2-5}%
\end{equation}
It is easy to check that $C_{1}(p,T,L_{1},\bar{\sigma})$ is continuous in $p$.
\end{proof}

\begin{remark}
If $p\geq2$, then
\[
\left(  \int_{t_{0}}^{t_{0}+\delta}|\hat{X}_{s}|^{2}ds\right)  ^{p/2}%
\leq \delta^{(p-2)/2}\int_{t_{0}}^{t_{0}+\delta}|\hat{X}_{s}|^{p}ds\leq
\delta^{(p-2)/2}\int_{t_{0}}^{t_{0}+\delta}\sup_{t\in \lbrack t_{0}%
,s]}\left \vert \hat{X}_{t}\right \vert ^{p}ds.
\]
Taking $t_{0}=0$ and $\delta=T$ in the proof of Theorem \ref{th-2-2} under
$p\geq2$, we obtain%
\begin{align*}
\mathbb{\hat{E}}\left[  \sup_{t\leq T}\left \vert \hat{X}_{t}\right \vert
^{p}\right]   &  \leq \lambda_{3}\int_{0}^{T}\mathbb{\hat{E}}\left[  \sup
_{t\in \leq s}\left \vert \hat{X}_{t}\right \vert ^{p}\right]  ds\\
&  \  \  \ +\lambda_{4}\mathbb{\hat{E}}\left[  \left(  \int_{0}^{T}(|\hat{b}%
_{t}|+|\hat{h}_{t}|)dt\right)  ^{p}+\left(  \int_{0}^{T}|\hat{\sigma}_{t}%
|^{2}dt\right)  ^{p/2}\right]  ,
\end{align*}
where%
\[
\lambda_{3}=6^{p-1}\left[  (1+\bar{\sigma}^{2p})(nL_{1})^{p}T^{p-1}%
+2C(p)(L_{1}n^{2}\bar{\sigma})^{p}T^{(p-2)/2}\right]  ,\text{ }\lambda
_{4}=6^{p-1}\left[  1+\bar{\sigma}^{2p}+2C(p)(n\bar{\sigma})^{p}\right]  .
\]
By the Gronwall inequality, we get%
\begin{equation}
C_{1}(p,T,L_{1},\bar{\sigma})=e^{\lambda_{3}T}\lambda_{4}. \label{e-2-7}%
\end{equation}

\end{remark}

The following theorem is Propositions 3.8 and 5.1 in \cite{HJPS1}.

\begin{theorem}
\label{th-2-3}Suppose assumptions (H1) and (H2) hold. For each $p\in(1,\beta)$
and $(x_{t}^{(i)})_{t\leq T}\in S_{G}^{p}(0,T;\mathbb{R}^{n})$, $i=1$, $2$.
Let $(Y_{t}^{(i)},Z_{t}^{(i)},K_{t}^{(i)})_{t\leq T}$ be the $L^{p}$-solution
of $G$-BSDE%
\[
dY_{t}^{(i)}=f(t,x_{t}^{(i)},Y_{t}^{(i)},Z_{t}^{(i)})dt+g(t,x_{t}^{(i)}%
,Y_{t}^{(i)},Z_{t}^{(i)})d\langle B\rangle_{t}+Z_{t}^{(i)}dB_{t}+dK_{t}%
^{(i)},\text{ }Y_{T}^{(i)}=\phi(x_{T}^{(i)}),
\]
for $i=1$, $2$. Then

\begin{description}
\item[(i)] there exists a deterministic function $C_{2}(p,T,L_{1},\bar{\sigma
},\underline{\sigma})>0$, which is continuous in $p$, such that%
\[
\left \vert \hat{Y}_{t}\right \vert ^{p}\leq C_{2}(p,T,L_{1},\bar{\sigma
},\underline{\sigma})\mathbb{\hat{E}}_{t}\left[  \left(  |\hat{\phi}_{T}%
|+\int_{t}^{T}(|\hat{f}_{s}|+|\hat{g}_{s}|)ds\right)  ^{p}\right]  ,
\]
where $\hat{Y}_{t}=Y_{t}^{(1)}-Y_{t}^{(2)}$, $\hat{\phi}_{T}=\phi(x_{T}%
^{(1)})-\phi(x_{T}^{(2)})$,%
\[
\hat{f}_{s}=f(s,x_{s}^{(1)},Y_{s}^{(2)},Z_{s}^{(2)})-f(s,x_{s}^{(2)}%
,Y_{s}^{(2)},Z_{s}^{(2)})\text{, }\hat{g}_{s}=g(s,x_{s}^{(1)},Y_{s}%
^{(2)},Z_{s}^{(2)})-g(s,x_{s}^{(2)},Y_{s}^{(2)},Z_{s}^{(2)}).
\]

\item[(ii)] there exists a deterministic function $C_{3}(p,T,L_{1},\bar
{\sigma},\underline{\sigma})>0$ such that%
\[
\mathbb{\hat{E}}\left[  \left(  \int_{0}^{T}|\hat{Z}_{t}|^{2}dt\right)
^{p/2}\right]  \leq C_{3}(p,T,L_{1},\bar{\sigma},\underline{\sigma})\left \{
\mathbb{\hat{E}}\left[  \sup_{t\leq T}\left \vert \hat{Y}_{t}\right \vert
^{p}\right]  +(\Lambda_{1}+\Lambda_{2})^{1/2}\left(  \mathbb{\hat{E}}\left[
\sup_{t\leq T}\left \vert \hat{Y}_{t}\right \vert ^{p}\right]  \right)
^{1/2}\right \}  ,
\]
where $\hat{Z}_{t}=Z_{t}^{(1)}-Z_{t}^{(2)}$,%
\[
\Lambda_{i}=\mathbb{\hat{E}}\left[  \sup_{t\leq T}|Y_{t}^{(i)}|^{p}\right]
+\mathbb{\hat{E}}\left[  \left(  \int_{0}^{T}(|f(s,x_{s}^{(i)}%
,0,0)|+|g(s,x_{s}^{(i)},0,0)|)ds\right)  ^{p}\right]  \text{ for }i=1,2.
\]

\end{description}
\end{theorem}

\begin{remark}
According to the proof of Proposition 5.1 in \cite{HJPS1}, we can deduce%
\begin{equation}
C_{2}(p,T,L_{1},\bar{\sigma},\underline{\sigma})=2^{p-1}\left[  1+(1+\bar
{\sigma}^{2})^{p}e^{pL_{1}(1+\bar{\sigma}^{2})T}\right]  e^{\lambda_{5}T},
\label{e-2-6}%
\end{equation}
where%
\[
\lambda_{5}=pL_{1}(1+\bar{\sigma}^{2})+\frac{1}{2}pL_{1}^{2}\bar{\sigma}%
^{2}(1+\underline{\sigma}^{-2})^{2}[(p-1)^{-1}\vee1].
\]

\end{remark}

\section{Existence and uniqueness of $L^{p}$-solution for $G$-FBSDEs}

For simplicity, we use $C_{1}(p)$ and $C_{2}(p)$ instead of $C_{1}%
(p,T,L_{1},\bar{\sigma})$ and $C_{2}(p,T,L_{1},\bar{\sigma},\underline{\sigma
})$ respectively in the following. The first main result in this section is
the existence and uniqueness of $L^{p}$-solution for $G$-FBSDE (\ref{e-2-3})
with $p\geq2$.

\begin{theorem}
\label{th-3-1}Suppose assumptions (H1) and (H2) hold. If $\beta>2$ and
\begin{equation}
\Lambda_{p}:=C_{1}(p)C_{2}(p)(nL_{2}L_{3})^{p}(T^{p}+T^{p/2})(1+T)^{p}<1
\label{e-3-1}%
\end{equation}
for some $p\in \lbrack2,\beta)$, then $G$-FBSDE (\ref{e-2-3}) has a unique
$L^{p}$-solution $(X,Y,Z,K)$.
\end{theorem}

\begin{proof}
We first prove the uniqueness. Let $(X,Y,Z,K)$ and $(X^{\prime},Y^{\prime
},Z^{\prime},K^{\prime})$ be two $L^{p}$-solutions of $G$-FBSDE (\ref{e-2-3}).
Set
\[
\hat{X}_{t}=X_{t}-X_{t}^{\prime},\text{ }\hat{Y}_{t}=Y_{t}-Y_{t}^{\prime
}\text{, }\hat{Z}_{t}=Z_{t}-Z_{t}^{\prime}\text{ for }t\in \lbrack0,T].
\]
By Theorem \ref{th-2-2}, we obtain
\begin{equation}
\mathbb{\hat{E}}\left[  \sup_{t\leq T}\left \vert \hat{X}_{t}\right \vert
^{p}\right]  \leq C_{1}(p)\mathbb{\hat{E}}\left[  \left(  \int_{0}^{T}%
(|\hat{b}_{t}|+|\hat{h}_{t}|)dt\right)  ^{p}+\left(  \int_{0}^{T}|\hat{\sigma
}_{t}|^{2}dt\right)  ^{p/2}\right]  , \label{e-3-1-1}%
\end{equation}
where $\hat{b}_{t}=b(t,X_{t}^{\prime},Y_{t})-b(t,X_{t}^{\prime},Y_{t}^{\prime
})$, $\hat{h}_{t}=h(t,X_{t}^{\prime},Y_{t})-h(t,X_{t}^{\prime},Y_{t}^{\prime
})$, $\hat{\sigma}_{t}=\sigma(t,X_{t}^{\prime},Y_{t})-\sigma(t,X_{t}^{\prime
},Y_{t}^{\prime})$. It follows from (H2) that%
\[
|\hat{b}_{t}|+|\hat{h}_{t}|+|\hat{\sigma}_{t}|\leq nL_{2}|\hat{Y}_{t}|.
\]
Thus we get%
\begin{equation}
\mathbb{\hat{E}}\left[  \sup_{t\leq T}\left \vert \hat{X}_{t}\right \vert
^{p}\right]  \leq C_{1}(p)(nL_{2})^{p}(T^{p-1}+T^{(p-2)/2})\int_{0}%
^{T}\mathbb{\hat{E}}[|\hat{Y}_{t}|^{p}]dt. \label{e-3-2}%
\end{equation}
By (i) of Theorem \ref{th-2-3}, we obtain%
\[
\left \vert \hat{Y}_{t}\right \vert ^{p}\leq C_{2}(p)\mathbb{\hat{E}}_{t}\left[
\left(  |\hat{\phi}_{T}|+\int_{t}^{T}(|\hat{f}_{s}|+|\hat{g}_{s}|)ds\right)
^{p}\right]  ,
\]
where $\hat{\phi}_{T}=\phi(X_{T})-\phi(X_{T}^{\prime})$,%
\[
\hat{f}_{s}=f(s,X_{s},Y_{s}^{\prime},Z_{s}^{\prime})-f(s,X_{s}^{\prime}%
,Y_{s}^{\prime},Z_{s}^{\prime})\text{, }\hat{g}_{s}=g(s,X_{s},Y_{s}^{\prime
},Z_{s}^{\prime})-g(s,X_{s}^{\prime},Y_{s}^{\prime},Z_{s}^{\prime}).
\]
From (H2), we have%
\[
|\hat{\phi}_{T}|\leq L_{3}|\hat{X}_{T}|\text{, }|\hat{f}_{s}|+|\hat{g}%
_{s}|\leq L_{3}|\hat{X}_{s}|.
\]
Then we deduce%
\begin{equation}
\mathbb{\hat{E}}[|\hat{Y}_{t}|^{p}]\leq C_{2}(p)L_{3}^{p}(1+T)^{p}%
\mathbb{\hat{E}}\left[  \sup_{s\leq T}\left \vert \hat{X}_{s}\right \vert
^{p}\right]  . \label{e-3-3}%
\end{equation}
It follows from (\ref{e-3-1}), (\ref{e-3-2}) and (\ref{e-3-3}) that%
\[
\mathbb{\hat{E}}\left[  \sup_{t\leq T}\left \vert \hat{X}_{t}\right \vert
^{p}\right]  \leq \Lambda_{p}\mathbb{\hat{E}}\left[  \sup_{t\leq T}\left \vert
\hat{X}_{t}\right \vert ^{p}\right]  ,
\]
which implies $\mathbb{\hat{E}}\left[  \sup_{t\leq T}\left \vert \hat{X}%
_{t}\right \vert ^{p}\right]  =0$. Then, by (\ref{e-3-3}), we obtain $\hat
{Y}_{t}=0$ q.s. Since $\hat{Y}_{t}$ is continuous in $t$, we can deduce%
\[
\sup_{t\leq T}\left \vert \hat{Y}_{t}\right \vert ^{p}=0\text{ q.s.,}%
\]
which implies $\mathbb{\hat{E}}\left[  \sup_{t\leq T}\left \vert \hat{Y}%
_{t}\right \vert ^{p}\right]  =0$. From (ii) of Theorem \ref{th-2-3}, we get%
\[
\mathbb{\hat{E}}\left[  \left(  \int_{0}^{T}|\hat{Z}_{t}|^{2}dt\right)
^{p/2}\right]  =0,
\]
which implies $K=K^{\prime}$ by $G$-FBSDE (\ref{e-2-3}). Thus the $L^{p}%
$-solution of $G$-FBSDE (\ref{e-2-3}) is unique.

Now we prove the existence. Set $X_{t}^{(0)}=x_{0}$ for $t\leq T$. Define
$(X^{(m)},Y^{(m)},Z^{(m)},K^{(m)})$, $m\geq1$, as follows:%
\begin{equation}
\left \{
\begin{array}
[c]{rl}%
dX_{t}^{(m)}= & b(t,X_{t}^{(m)},Y_{t}^{(m)})dt+h(t,X_{t}^{(m)},Y_{t}%
^{(m)})d\langle B\rangle_{t}+\sigma(t,X_{t}^{(m)},Y_{t}^{(m)})dB_{t},\\
dY_{t}^{(m)}= & f(t,X_{t}^{(m-1)},Y_{t}^{(m)},Z_{t}^{(m)})dt+g(t,X_{t}%
^{(m-1)},Y_{t}^{(m)},Z_{t}^{(m)})d\langle B\rangle_{t}+Z_{t}^{(m)}%
dB_{t}+dK_{t}^{(m)},\\
X_{0}^{(m)}= & x_{0}\in \mathbb{R}^{n},\text{ }Y_{T}^{(m)}=\phi(X_{T}^{(m-1)}).
\end{array}
\right.  \label{e-3-4}%
\end{equation}
For $m=1$, we first slove $G$-BSDE in (\ref{e-3-4}) to get $(Y^{(1)}%
,Z^{(1)},K^{(1)})$. Since $X^{(0)}\in S_{G}^{\alpha}(0,T;\mathbb{R}^{n})$ for
each $\alpha<\beta$, we obtain%
\[
Y^{(1)}\in S_{G}^{\alpha}(0,T)\text{, }Z^{(1)}\in M_{G}^{2,\alpha}(0,T)\text{,
}K_{T}^{(1)}\in L_{G}^{\alpha}(\Omega_{T}),
\]
for each $\alpha<\beta$ by Theorem 4.1 in \cite{HJPS1}. We then slove $G$-SDE
in (\ref{e-3-4}) to get $X^{(1)}$. Obviously, $X^{(1)}\in S_{G}^{\alpha
}(0,T;\mathbb{R}^{n})$ for each $\alpha<\beta$ by Theorem \ref{th-2-2}.
Continuing this process, we can get%
\[
X^{(m)}\in S_{G}^{\alpha}(0,T;\mathbb{R}^{n})\text{, }Y^{(m)}\in S_{G}%
^{\alpha}(0,T)\text{, }Z^{(m)}\in M_{G}^{2,\alpha}(0,T)\text{, }K_{T}^{(m)}\in
L_{G}^{\alpha}(\Omega_{T}),
\]
for each $\alpha<\beta$ and $m\geq1$. Since $\Lambda_{p}$ is continuous in $p$
and $\Lambda_{p}<1$, there exists a $p^{\prime}\in(p,\beta)$ such that
$\Lambda_{p^{\prime}}<1$. Set%
\[
\hat{X}^{(m)}=X^{(m)}-X^{(m-1)}\text{ for }m\geq1\text{, }\hat{Y}%
^{(m)}=Y^{(m)}-Y^{(m-1)}\text{ and }\hat{Z}^{(m)}=Z^{(m)}-Z^{(m-1)}\text{ for
}m\geq2.
\]
By Theorem \ref{th-2-2}, we get, for $m\geq2$,%
\[
\mathbb{\hat{E}}\left[  \sup_{t\leq T}\left \vert \hat{X}_{t}^{(m)}\right \vert
^{p^{\prime}}\right]  \leq C_{1}(p^{\prime})\mathbb{\hat{E}}\left[  \left(
\int_{0}^{T}(|\hat{b}_{t}^{(m)}|+|\hat{h}_{t}^{(m)}|)dt\right)  ^{p^{\prime}%
}+\left(  \int_{0}^{T}|\hat{\sigma}_{t}^{(m)}|^{2}dt\right)  ^{p^{\prime}%
/2}\right]  ,
\]
where $\hat{b}_{t}^{(m)}=b(t,X_{t}^{(m-1)},Y_{t}^{(m)})-b(t,X_{t}%
^{(m-1)},Y_{t}^{(m-1)})$, $\hat{h}_{t}^{(m)}=h(t,X_{t}^{(m-1)},Y_{t}%
^{(m)})-h(t,X_{t}^{(m-1)},Y_{t}^{(m-1)})$, $\hat{\sigma}_{t}^{(m)}%
=\sigma(t,X_{t}^{(m-1)},Y_{t}^{(m)})-\sigma(t,X_{t}^{(m-1)},Y_{t}^{(m-1)})$.
Similar to the proof of (\ref{e-3-2}), we obtain%
\begin{equation}
\mathbb{\hat{E}}\left[  \sup_{t\leq T}\left \vert \hat{X}_{t}^{(m)}\right \vert
^{p^{\prime}}\right]  \leq C_{1}(p^{\prime})(nL_{2})^{p^{\prime}}%
(T^{p^{\prime}-1}+T^{(p^{\prime}-2)/2})\int_{0}^{T}\mathbb{\hat{E}}[|\hat
{Y}_{t}^{(m)}|^{p^{\prime}}]dt. \label{e-3-5}%
\end{equation}
It follows from (i) of Theorem \ref{th-2-3} that, for $m\geq2$,
\[
\left \vert \hat{Y}_{t}^{(m)}\right \vert ^{p^{\prime}}\leq C_{2}(p^{\prime
})\mathbb{\hat{E}}_{t}\left[  \left(  |\hat{\phi}_{T}^{(m)}|+\int_{t}%
^{T}(|\hat{f}_{s}^{(m)}|+|\hat{g}_{s}^{(m)}|)ds\right)  ^{p^{\prime}}\right]
,
\]
where $\hat{\phi}_{T}^{(m)}=\phi(X_{T}^{(m-1)})-\phi(X_{T}^{(m-2)})$,%
\[%
\begin{array}
[c]{l}%
\hat{f}_{s}^{(m)}=f(s,X_{s}^{(m-1)},Y_{s}^{(m-1)},Z_{s}^{(m-1)})-f(s,X_{s}%
^{(m-2)},Y_{s}^{(m-1)},Z_{s}^{(m-1)})\text{,}\\
\hat{g}_{s}^{(m)}=g(s,X_{s}^{(m-1)},Y_{s}^{(m-1)},Z_{s}^{(m-1)})-g(s,X_{s}%
^{(m-2)},Y_{s}^{(m-1)},Z_{s}^{(m-1)}).
\end{array}
\]
Similar to the proof of (\ref{e-3-3}), we get%
\begin{equation}
\mathbb{\hat{E}}\left[  \left \vert \hat{Y}_{t}^{(m)}\right \vert ^{p^{\prime}%
}\right]  \leq C_{2}(p^{\prime})L_{3}^{p^{\prime}}(1+T)^{p^{\prime}%
}\mathbb{\hat{E}}\left[  \sup_{s\leq T}\left \vert \hat{X}_{s}^{(m-1)}%
\right \vert ^{p^{\prime}}\right]  . \label{e-3-6}%
\end{equation}
By (\ref{e-3-5}) and (\ref{e-3-6}), we deduce%
\[
\mathbb{\hat{E}}\left[  \sup_{t\leq T}\left \vert \hat{X}_{t}^{(m)}\right \vert
^{p^{\prime}}\right]  \leq \Lambda_{p^{\prime}}\mathbb{\hat{E}}\left[
\sup_{t\leq T}\left \vert \hat{X}_{t}^{(m-1)}\right \vert ^{p^{\prime}}\right]
\text{ for }m\geq2,
\]
which implies%
\[
\mathbb{\hat{E}}\left[  \sup_{t\leq T}\left \vert \hat{X}_{t}^{(m)}\right \vert
^{p^{\prime}}\right]  \leq \Lambda_{p^{\prime}}^{m-1}\mathbb{\hat{E}}\left[
\sup_{t\leq T}\left \vert \hat{X}_{t}^{(1)}\right \vert ^{p^{\prime}}\right]
\text{ for }m\geq1.
\]
For each $N$, $k\geq1$, we obtain%
\begin{align*}
\left(  \mathbb{\hat{E}}\left[  \sup_{t\leq T}\left \vert X_{t}^{(N+k)}%
-X_{t}^{(N)}\right \vert ^{p^{\prime}}\right]  \right)  ^{1/p^{\prime}}  &
\leq \sum_{m=N+1}^{\infty}\left(  \mathbb{\hat{E}}\left[  \sup_{t\leq
T}\left \vert \hat{X}_{t}^{(m)}\right \vert ^{p^{\prime}}\right]  \right)
^{1/p^{\prime}}\\
&  \leq(1-\Lambda_{p^{\prime}}^{1/p^{\prime}})^{-1}\Lambda_{p^{\prime}%
}^{N/p^{\prime}}\left(  \mathbb{\hat{E}}\left[  \sup_{t\leq T}\left \vert
\hat{X}_{t}^{(1)}\right \vert ^{p^{\prime}}\right]  \right)  ^{1/p^{\prime}},
\end{align*}
which tends to $0$ as $N\rightarrow \infty$. Thus there exists a $X\in
S_{G}^{p^{\prime}}(0,T;\mathbb{R}^{n})$ such that%
\begin{equation}
\mathbb{\hat{E}}\left[  \sup_{t\leq T}\left \vert X_{t}^{(m)}-X_{t}\right \vert
^{p^{\prime}}\right]  \rightarrow0\text{ as }m\rightarrow \infty. \label{e-3-7}%
\end{equation}
For each $N$, $k\geq1$, similar to the proof of (\ref{e-3-6}), we can deduce%
\begin{equation}
\left \vert Y_{t}^{(N+k)}-Y_{t}^{(N)}\right \vert ^{p}\leq C_{2}(p)L_{3}%
^{p}(1+T)^{p}\mathbb{\hat{E}}_{t}\left[  \sup_{s\leq T}\left \vert
X_{s}^{(N+k-1)}-X_{s}^{(N-1)}\right \vert ^{p}\right]  . \label{e-3-8}%
\end{equation}
By Doob's inequality for $G$-martingale (see \cite{STZ, Song11}), we have%
\begin{equation}
\mathbb{\hat{E}}\left[  \sup_{t\leq T}\mathbb{\hat{E}}_{t}\left[  \sup_{s\leq
T}\left \vert X_{s}^{(N+k-1)}-X_{s}^{(N-1)}\right \vert ^{p}\right]  \right]
\leq \frac{p^{\prime}}{p^{\prime}-p}\left(  \mathbb{\hat{E}}\left[  \sup_{s\leq
T}\left \vert X_{s}^{(N+k-1)}-X_{s}^{(N-1)}\right \vert ^{p^{\prime}}\right]
\right)  ^{p/p^{\prime}}. \label{e-3-9}%
\end{equation}
It follows from (\ref{e-3-7}), (\ref{e-3-8}) and (\ref{e-3-9}) that%
\[
\mathbb{\hat{E}}\left[  \sup_{t\leq T}\left \vert Y_{t}^{(N+k)}-Y_{t}%
^{(N)}\right \vert ^{p}\right]  \rightarrow0\text{ as }N\rightarrow \infty.
\]
Thus there exists a $Y\in S_{G}^{p}(0,T)$ such that%
\begin{equation}
\mathbb{\hat{E}}\left[  \sup_{t\leq T}\left \vert Y_{t}^{(m)}-Y_{t}\right \vert
^{p}\right]  \rightarrow0\text{ as }m\rightarrow \infty. \label{e-3-10}%
\end{equation}
Noting that $\sup_{m\geq1}\mathbb{\hat{E}}\left[  \sup_{t\leq T}(|X_{t}%
^{(m)}|+|Y_{t}^{(m)}|)^{p}\right]  <\infty$, by (ii) of Theorem \ref{th-2-3},
we get%
\[
\mathbb{\hat{E}}\left[  \left(  \int_{0}^{T}|Z_{t}^{(N+k)}-Z_{t}^{(N)}%
|^{2}dt\right)  ^{p/2}\right]  \rightarrow0\text{ as }N\rightarrow \infty.
\]
Thus there exists a $Z\in M_{G}^{2,p}(0,T)$ such that%
\begin{equation}
\mathbb{\hat{E}}\left[  \left(  \int_{0}^{T}|Z_{t}^{(m)}-Z_{t}|^{2}dt\right)
^{p/2}\right]  \rightarrow0\text{ as }m\rightarrow \infty. \label{e-3-11}%
\end{equation}
From (\ref{e-2-2}), we obtain%
\begin{align*}
\mathbb{\hat{E}}\left[  \sup_{t\leq T}\left \vert \int_{t}^{T}Z_{s}^{(m)}%
dB_{s}-\int_{t}^{T}Z_{s}dB_{s}\right \vert ^{p}\right]   &  \leq2^{p}%
\mathbb{\hat{E}}\left[  \sup_{t\leq T}\left \vert \int_{0}^{t}Z_{s}^{(m)}%
dB_{s}-\int_{0}^{t}Z_{s}dB_{s}\right \vert ^{p}\right] \\
&  \leq2^{p}\bar{\sigma}^{p}C(p)\mathbb{\hat{E}}\left[  \left(  \int_{0}%
^{T}|Z_{t}^{(m)}-Z_{t}|^{2}dt\right)  ^{p/2}\right] \\
&  \rightarrow0\text{ as }m\rightarrow \infty.
\end{align*}
Since%
\begin{align*}
&  \sup_{t\leq T}\left \vert \int_{t}^{T}f(s,X_{s}^{(m-1)},Y_{s}^{(m)}%
,Z_{s}^{(m)})ds-\int_{t}^{T}f(s,X_{s},Y_{s},Z_{s})ds\right \vert ^{p}\\
&  \leq \left(  \int_{0}^{T}|f(s,X_{s}^{(m-1)},Y_{s}^{(m)},Z_{s}^{(m)}%
)-f(s,X_{s},Y_{s},Z_{s})|ds\right)  ^{p}\\
&  \leq3^{p-1}L_{3}^{p}T^{p}\sup_{s\leq T}\left \vert X_{s}^{(m-1)}%
-X_{s}\right \vert ^{p}+3^{p-1}L_{1}^{p}T^{p}\sup_{s\leq T}\left \vert
Y_{s}^{(m)}-Y_{s}\right \vert ^{p}+3^{p-1}L_{1}^{p}T^{p/2}\left(  \int_{0}%
^{T}|Z_{s}^{(m)}-Z_{s}|^{2}ds\right)  ^{p/2},
\end{align*}
we get%
\[
\mathbb{\hat{E}}\left[  \sup_{t\leq T}\left \vert \int_{t}^{T}f(s,X_{s}%
^{(m-1)},Y_{s}^{(m)},Z_{s}^{(m)})ds-\int_{t}^{T}f(s,X_{s},Y_{s},Z_{s}%
)ds\right \vert ^{p}\right]  \rightarrow0
\]
as $m\rightarrow \infty$ by (\ref{e-3-7}), (\ref{e-3-10}) and (\ref{e-3-11}).
Similarly, we can obtain%
\[
\mathbb{\hat{E}}\left[  \sup_{t\leq T}\left \vert \int_{t}^{T}g(s,X_{s}%
^{(m-1)},Y_{s}^{(m)},Z_{s}^{(m)})d\langle B\rangle_{s}-\int_{t}^{T}%
g(s,X_{s},Y_{s},Z_{s})d\langle B\rangle_{s}\right \vert ^{p}\right]
\rightarrow0,
\]%
\[
\mathbb{\hat{E}}\left[  \sup_{t\leq T}\left(  \left \vert \int_{0}%
^{t}(b(s,X_{s}^{(m)},Y_{s}^{(m)})-b(s,X_{s},Y_{s}))ds\right \vert +\left \vert
\int_{0}^{t}(h(s,X_{s}^{(m)},Y_{s}^{(m)})-h(s,X_{s},Y_{s}))d\langle
B\rangle_{s}\right \vert \right)  ^{p}\right]  \rightarrow0
\]
and%
\[
\mathbb{\hat{E}}\left[  \sup_{t\leq T}\left \vert \int_{0}^{t}(\sigma
(s,X_{s}^{(m)},Y_{s}^{(m)})-\sigma(s,X_{s},Y_{s}))dB_{s}\right \vert
^{p}\right]  \rightarrow0
\]
as $m\rightarrow \infty$. Set%
\[
K_{t}=Y_{t}-Y_{0}-\int_{0}^{t}f(s,X_{s},Y_{s},Z_{s})ds-\int_{0}^{t}%
g(s,X_{s},Y_{s},Z_{s})d\langle B\rangle_{s}-\int_{0}^{t}Z_{s}dB_{s}%
\]
for $t\in \lbrack0,T]$. It is clear that%
\[
\mathbb{\hat{E}}\left[  \sup_{t\leq T}\left \vert K_{t}^{(m)}-K_{t}\right \vert
^{p}\right]  \rightarrow0\text{ as }m\rightarrow \infty.
\]
Thus we can easily deduce that $K$ is a non-increasing $G$-martingale with
$K_{0}=0$ and $K_{T}\in L_{G}^{p}(\Omega_{T})$. Taking $m\rightarrow \infty$ in
(\ref{e-3-4}), we obtain that $(X,Y,Z,K)$ is an $L^{p}$-solution of $G$-FBSDE
(\ref{e-2-3}).
\end{proof}

\begin{remark}
For each fixed $\bar{\sigma}>\underline{\sigma}>0$, $T>0$, $L_{1}>0$ and
$p\in \lbrack2,\beta)$, it is easy to deduce from (\ref{e-3-1}) that there
exists a $\delta>0$ satisfying $\Lambda_{p}<1$ for each%
\begin{equation}
L_{2}L_{3}<\delta. \label{e-3-12}%
\end{equation}
The condition (\ref{e-3-12}) is called weakly coupling condition for $G$-FBSDE
(\ref{e-2-3}) (see \cite{P-T} for classical FBSDE).
\end{remark}

Now we consider the $L^{p}$-solution for $G$-FBSDE (\ref{e-2-3}) with
$p\in(1,2)$.

\begin{theorem}
\label{th-3-2}Suppose assumptions (H1) and (H2) hold. If $\sigma(\cdot)$ does
not depend on $y$ and%
\begin{equation}
\tilde{\Lambda}_{p}:=C_{1}(p)C_{2}(p)(nL_{2}L_{3})^{p}T^{p}(1+T)^{p}<1
\label{e-3-13}%
\end{equation}
for some $p\in(1,2\wedge \beta)$, then $G$-FBSDE (\ref{e-2-3}) has a unique
$L^{p}$-solution $(X,Y,Z,K)$.
\end{theorem}

\begin{proof}
The proof is similar to the proof of Theorem \ref{th-3-1}. We omit it.
\end{proof}

\begin{remark}
If $\sigma(\cdot)$ contains $y$ and $p\in(1,2\wedge \beta)$, then $p/2<1$ and
we can not get%
\[
\left(  \int_{0}^{T}|\hat{Y}_{t}|^{2}dt\right)  ^{p/2}\leq C\int_{0}^{T}%
|\hat{Y}_{t}|^{p}dt
\]
in (\ref{e-3-1-1}), where $C>0$ is a constant independent of $\hat{Y}$. Thus
we need the assumption that $\sigma(\cdot)$ is independent of $y$ for $p<2$..
\end{remark}

The following proposition is the estimates for $G$-FBSDE (\ref{e-2-3}).

\begin{proposition}
\label{pr-3-3}Suppose that $b^{(i)}(\cdot)$, $h^{(i)}(\cdot)$, $\sigma
^{(i)}(\cdot)$, $f_{i}(\cdot)$, $g_{i}(\cdot)$, $\phi_{i}(\cdot)$ satisfy
assumptions (H1) and (H2) for $i=1$, $2$. For each fixed $p\in(1,\beta)$, let
$(X^{(i)},Y^{(i)},Z^{(i)},K^{(i)})$ be the $L^{p}$-solution of $G$-FBSDE%
\[
\left \{
\begin{array}
[c]{rl}%
dX_{t}^{(i)}= & b^{(i)}(t,X_{t}^{(i)},Y_{t}^{(i)})dt+h^{(i)}(t,X_{t}%
^{(i)},Y_{t}^{(i)})d\langle B\rangle_{t}+\sigma^{(i)}(t,X_{t}^{(i)}%
,Y_{t}^{(i)})dB_{t},\\
dY_{t}^{(i)}= & f_{i}(t,X_{t}^{(i)},Y_{t}^{(i)},Z_{t}^{(i)})dt+g_{i}%
(t,X_{t}^{(i)},Y_{t}^{(i)},Z_{t}^{(i)})d\langle B\rangle_{t}+Z_{t}^{(i)}%
dB_{t}+dK_{t}^{(i)},\\
X_{0}^{(i)}= & x_{i}\in \mathbb{R}^{n},\text{ }Y_{T}^{(i)}=\phi_{i}(X_{T}%
^{(i)}),
\end{array}
\right.
\]
for $i=1$, $2$. We have the following estimates.

\begin{description}
\item[(i)] If $p\geq2$ and $\Lambda_{p}$ defined in (\ref{e-3-1}) satisfies
$\Lambda_{p}<1$, then there exists a constant $C_{4}$ depending on $p$, $T$,
$L_{1}$, $L_{2}$, $L_{3}$, $\bar{\sigma}$ and $\underline{\sigma}$ such that
\begin{equation}
\mathbb{\hat{E}}\left[  \sup_{t\leq T}\left \vert \hat{X}_{t}\right \vert
^{p}\right]  \leq C_{4}\mathbb{\hat{E}}\left[  \left(  |\hat{x}|+|\hat{\phi
}_{T}|+\int_{0}^{T}(|\hat{b}_{t}|+|\hat{h}_{t}|+|\hat{f}_{t}|+|\hat{g}%
_{t}|)dt\right)  ^{p}+\left(  \int_{0}^{T}|\hat{\sigma}_{t}|^{2}dt\right)
^{p/2}\right]  , \label{e-3-14}%
\end{equation}
where $\hat{X}_{t}=X_{t}^{(1)}-X_{t}^{(2)}$, $\hat{x}=x_{1}-x_{2}$, $\hat
{\phi}_{T}=\phi_{1}(X_{T}^{(2)})-\phi_{2}(X_{T}^{(2)})$, $\hat{b}_{t}%
=b^{(1)}(t,X_{t}^{(2)},Y_{t}^{(2)})-b^{(2)}(t,X_{t}^{(2)},Y_{t}^{(2)})$,
$\hat{h}_{t}=h^{(1)}(t,X_{t}^{(2)},Y_{t}^{(2)})-h^{(2)}(t,X_{t}^{(2)}%
,Y_{t}^{(2)})$, $\hat{\sigma}_{t}=\sigma^{(1)}(t,X_{t}^{(2)},Y_{t}%
^{(2)})-\sigma^{(2)}(t,X_{t}^{(2)},Y_{t}^{(2)})$,
\[
\hat{f}_{t}=f_{1}(t,X_{t}^{(2)},Y_{t}^{(2)},Z_{t}^{(2)})-f_{2}(t,X_{t}%
^{(2)},Y_{t}^{(2)},Z_{t}^{(2)}),\text{ }\hat{g}_{t}=g_{1}(t,X_{t}^{(2)}%
,Y_{t}^{(2)},Z_{t}^{(2)})-g_{2}(t,X_{t}^{(2)},Y_{t}^{(2)},Z_{t}^{(2)}).
\]

\item[(ii)] If $p\in(1,2)$, $\sigma(\cdot)$ does not depend on $y$ and
$\tilde{\Lambda}_{p}$ defined in (\ref{e-3-13}) satisfies $\tilde{\Lambda}%
_{p}<1$, then there exists a constant $C_{5}$ depending on $p$, $T$, $L_{1}$,
$L_{2}$, $L_{3}$, $\bar{\sigma}$ and $\underline{\sigma}$ such that%
\begin{equation}
\mathbb{\hat{E}}\left[  \sup_{t\leq T}\left \vert \hat{X}_{t}\right \vert
^{p}\right]  \leq C_{5}\mathbb{\hat{E}}\left[  \left(  |\hat{x}|+|\hat{\phi
}_{T}|+\int_{0}^{T}(|\hat{b}_{t}|+|\hat{h}_{t}|+|\hat{f}_{t}|+|\hat{g}%
_{t}|)dt\right)  ^{p}+\left(  \int_{0}^{T}|\hat{\sigma}_{t}|^{2}dt\right)
^{p/2}\right]  , \label{e-3-15}%
\end{equation}
where $\hat{\sigma}_{t}=\sigma^{(1)}(t,X_{t}^{(2)})-\sigma^{(2)}(t,X_{t}%
^{(2)})$, $\hat{X}_{t}$, $\hat{x}$, $\hat{\phi}_{T}$, $\hat{b}_{t}$, $\hat
{h}_{t}$, $\hat{f}_{t}$ and $\hat{g}_{t}$ are the same as (i).
\end{description}
\end{proposition}

\begin{proof}
We only prove (i). The proof of (ii) is similar. For each $a_{1}>0$ and
$a_{2}>0$, by the mean value theorem, we have%
\[
(a_{1}+a_{2})^{p}-a_{1}^{p}=p(a_{1}+\theta a_{2})^{p-1}a_{2}\leq
p2^{p-1}(a_{1}^{p-1}a_{2}+a_{2}^{p}),
\]
where $\theta \in(0,1)$. From this, we can deduce%
\begin{equation}
(a_{1}+a_{2})^{p}\leq(1+\varepsilon)a_{1}^{p}+C(p,\varepsilon)a_{2}^{p}\text{
for each }\varepsilon>0, \label{e-3-16}%
\end{equation}
where%
\[
C(p,\varepsilon)=p2^{p-1}+p^{p-1}2^{(p-1)p}\varepsilon^{-(p-1)}.
\]
Set $\bar{X}_{t}^{(i)}=X_{t}^{(i)}-x_{i}$ for $i=1$, $2$, and $\tilde{X}%
_{t}=\bar{X}_{t}^{(1)}-\bar{X}_{t}^{(2)}$. It is easy to check that $(\bar
{X}^{(i)},Y^{(i)},Z^{(i)},K^{(i)})$ satisfies the $G$-FBSDE%
\[
\left \{
\begin{array}
[c]{rl}%
d\bar{X}_{t}^{(i)}= & b^{(i)}(t,\bar{X}_{t}^{(i)}+x_{i},Y_{t}^{(i)}%
)dt+h^{(i)}(t,\bar{X}_{t}^{(i)}+x_{i},Y_{t}^{(i)})d\langle B\rangle_{t}%
+\sigma^{(i)}(t,\bar{X}_{t}^{(i)}+x_{i},Y_{t}^{(i)})dB_{t},\\
dY_{t}^{(i)}= & f_{i}(t,\bar{X}_{t}^{(i)}+x_{i},Y_{t}^{(i)},Z_{t}%
^{(i)})dt+g_{i}(t,\bar{X}_{t}^{(i)}+x_{i},Y_{t}^{(i)},Z_{t}^{(i)})d\langle
B\rangle_{t}+Z_{t}^{(i)}dB_{t}+dK_{t}^{(i)},\\
\bar{X}_{0}^{(i)}= & 0\in \mathbb{R}^{n},\text{ }Y_{T}^{(i)}=\phi_{i}(\bar
{X}_{T}^{(i)}+x_{i}),
\end{array}
\right.
\]
for $i=1$, $2$. Similar to the proof of Theorem \ref{th-2-2}, we have%
\[
\mathbb{\hat{E}}\left[  \sup_{t\leq T}\left \vert \tilde{X}_{t}\right \vert
^{p}\right]  \leq C_{1}(p)\mathbb{\hat{E}}\left[  \left(  \int_{0}^{T}%
(|\tilde{b}_{t}|+|\tilde{h}_{t}|)dt\right)  ^{p}+\left(  \int_{0}^{T}%
|\tilde{\sigma}_{t}|^{2}dt\right)  ^{p/2}\right]  ,
\]
where $\tilde{b}_{t}=b^{(1)}(t,\bar{X}_{t}^{(2)}+x_{1},Y_{t}^{(1)}%
)-b^{(2)}(t,\bar{X}_{t}^{(2)}+x_{2},Y_{t}^{(2)})$, $\tilde{h}_{t}%
=h^{(1)}(t,\bar{X}_{t}^{(2)}+x_{1},Y_{t}^{(1)})-h^{(2)}(t,\bar{X}_{t}%
^{(2)}+x_{2},Y_{t}^{(2)})$, $\tilde{\sigma}_{t}=\sigma^{(1)}(t,\bar{X}%
_{t}^{(2)}+x_{1},Y_{t}^{(1)})-\sigma^{(2)}(t,\bar{X}_{t}^{(2)}+x_{2}%
,Y_{t}^{(2)})$. From (H2), it is easy to verify that%
\[
|\tilde{b}_{t}|+|\tilde{h}_{t}|\leq nL_{2}\left \vert \hat{Y}_{t}\right \vert
+nL_{1}|\hat{x}|+|\hat{b}_{t}|+|\hat{h}_{t}|,\text{ }|\tilde{\sigma}_{t}|\leq
nL_{2}\left \vert \hat{Y}_{t}\right \vert +nL_{1}|\hat{x}|+|\hat{\sigma}_{t}|,
\]
where $\hat{Y}_{t}=Y_{t}^{(1)}-Y_{t}^{(2)}$. Similar to (\ref{e-3-2}), by
(\ref{e-3-16}), we obtain, for each $\varepsilon>0$,
\begin{align*}
\mathbb{\hat{E}}\left[  \sup_{t\leq T}\left \vert \tilde{X}_{t}\right \vert
^{p}\right]   &  \leq(1+\varepsilon)C_{1}(p)(nL_{2})^{p}(T^{p-1}%
+T^{(p-2)/2})\int_{0}^{T}\mathbb{\hat{E}}[|\hat{Y}_{t}|^{p}]dt\\
&  \  \  \ +C_{6}\mathbb{\hat{E}}\left[  \left(  |\hat{x}|+\int_{0}^{T}(|\hat
{b}_{t}|+|\hat{h}_{t}|)dt\right)  ^{p}+\left(  \int_{0}^{T}|\hat{\sigma}%
_{t}|^{2}dt\right)  ^{p/2}\right]  ,
\end{align*}
where the constant $C_{6}>0$ depends on $p$, $T$, $L_{1}$, $\bar{\sigma}$ and
$\varepsilon$. Similar to (\ref{e-3-3}), we can get, for each $\varepsilon
>0$,
\begin{align*}
\mathbb{\hat{E}}[|\hat{Y}_{t}|^{p}]  &  \leq(1+\varepsilon)C_{2}(p)L_{3}%
^{p}(1+T)^{p}\mathbb{\hat{E}}\left[  \sup_{s\leq T}\left \vert \tilde{X}%
_{s}\right \vert ^{p}\right] \\
&  \  \  \ +C_{7}\mathbb{\hat{E}}\left[  \left(  |\hat{x}|+|\hat{\phi}_{T}%
|+\int_{0}^{T}(|\hat{f}_{t}|+|\hat{g}_{t}|)dt\right)  ^{p}\right]  ,
\end{align*}
where the constant $C_{7}>0$ depends on $p$, $T$, $L_{1}$, $L_{3}$,
$\bar{\sigma}$, $\underline{\sigma}$ and $\varepsilon$. Thus we obtain%
\[
\lbrack1-(1+\varepsilon)\Lambda_{p}]\mathbb{\hat{E}}\left[  \sup_{t\leq
T}\left \vert \tilde{X}_{t}\right \vert ^{p}\right]  \leq C_{8}\mathbb{\hat{E}%
}\left[  \left(  |\hat{x}|+|\hat{\phi}_{T}|+\int_{0}^{T}(|\hat{b}_{t}%
|+|\hat{h}_{t}|+|\hat{f}_{t}|+|\hat{g}_{t}|)dt\right)  ^{p}+\left(  \int
_{0}^{T}|\hat{\sigma}_{t}|^{2}dt\right)  ^{p/2}\right]  ,
\]
where the constant $C_{8}>0$ depends on $p$, $T$, $L_{1}$, $L_{2}$, $L_{3}$,
$\bar{\sigma}$, $\underline{\sigma}$ and $\varepsilon$. Since $\Lambda_{p}<1$,
we can take $\varepsilon_{0}>0$ such that $(1+\varepsilon_{0})\Lambda_{p}<1$.
Note that $|\hat{X}_{t}|^{p}\leq2^{p-1}(\left \vert \tilde{X}_{t}\right \vert
^{p}+|\hat{x}|^{p})$, then we obtain (\ref{e-3-14}).
\end{proof}

\section{Comparison theorem for $G$-FBSDEs}

For simplicity, we only study the comparison theorem for $p=2$. The results
for $p\not =2$ are similar. Consider the following $G$-FBSDEs:%
\begin{equation}
\left \{
\begin{array}
[c]{rl}%
dX_{t}^{(i)}= & b(t,X_{t}^{(i)},Y_{t}^{(i)})dt+h(t,X_{t}^{(i)},Y_{t}%
^{(i)})d\langle B\rangle_{t}+\sigma(t,X_{t}^{(i)},Y_{t}^{(i)})dB_{t},\\
dY_{t}^{(i)}= & f(t,X_{t}^{(i)},Y_{t}^{(i)},Z_{t}^{(i)})dt+g(t,X_{t}%
^{(i)},Y_{t}^{(i)},Z_{t}^{(i)})d\langle B\rangle_{t}+Z_{t}^{(i)}dB_{t}%
+dK_{t}^{(i)},\\
X_{0}^{(i)}= & x_{0}\in \mathbb{R}^{n},\text{ }Y_{T}^{(i)}=\phi_{i}(X_{T}%
^{(i)}),\text{ }i=1,2.
\end{array}
\right.  \label{e-4-1}%
\end{equation}

\begin{theorem}
\label{th-4-1}Suppose that assumptions (H1) and (H2) hold for $i=1$, $2$ with
$\beta>2$. Then there exists a $\delta>0$ depending on $n$, $T$, $L_{1}$,
$\bar{\sigma}$ and $\underline{\sigma}$ such that the following results hold.

\begin{description}
\item[(i)] If $L_{2}L_{3}<\delta$, then $G$-FBSDE (\ref{e-4-1}) has a unique
$L^{2}$-solution $(X^{(i)},Y^{(i)},Z^{(i)},K^{(i)})$ for $i=1$, $2$.

\item[(ii)] If $L_{2}L_{3}<\delta$ and $\phi_{1}(X_{T}^{(2)})\geq \phi
_{2}(X_{T}^{(2)})$ (resp. $\phi_{1}(X_{T}^{(1)})\geq \phi_{2}(X_{T}^{(1)})$),
then we have $Y_{0}^{(1)}\geq Y_{0}^{(2)}$.
\end{description}
\end{theorem}

\begin{proof}
From the definition of $\Lambda_{p}$ in (\ref{e-3-1}) for $p\geq2$, it is easy
to deduce that there exists a $\delta_{1}>0$ depending on $n$, $T$, $L_{1}$,
$\bar{\sigma}$ and $\underline{\sigma}$ satisfying $\Lambda_{2}<1$. By Theorem
\ref{th-3-1}, we obtain (i) under the assumption $L_{2}L_{3}<\delta_{1}$.

We only prove the case $\phi_{1}(X_{T}^{(2)})\geq \phi_{2}(X_{T}^{(2)})$ for
(ii). The proof for $\phi_{1}(X_{T}^{(1)})\geq \phi_{2}(X_{T}^{(1)})$ is
similar. Under the assumption $L_{2}L_{3}<\delta_{1}$, it is clear that
$(X^{(i)},Y^{(i)},Z^{(i)},K^{(i)})$ is the $L^{2}$-solution of $G$-FBSDE
(\ref{e-4-1}) for $i=1$, $2$ under each $P\in \mathcal{P}$, where $\mathcal{P}$
is defined in Theorem \ref{th-2-1}. Since $\mathcal{P}$ is weakly compact and
$\mathbb{\hat{E}}[K_{T}^{(2)}]=0$ with $K_{T}^{(2)}\leq0$, there exists a
$P^{\ast}\in \mathcal{P}$ such that $K_{T}^{(2)}=0$ $P^{\ast}$-a.s. Noting that
$K^{(2)}$ is a non-increasing with $K_{0}^{(2)}=0$, we obtain $K^{(2)}=0$
under $P^{\ast}$. By (\ref{e-2-1}), we know that $d\langle B\rangle_{t}%
=\gamma_{t}dt$ q.s. with $\gamma_{t}\in \lbrack \underline{\sigma}^{2}%
,\bar{\sigma}^{2}]$.

Set $X_{t}^{(i)}=(X_{1,t}^{(i)},\ldots,X_{n,t}^{(i)})^{T}$ for $i=1$, $2$,
$\hat{X}_{t}=(\hat{X}_{1,t},\ldots,\hat{X}_{n,t})^{T}=X_{t}^{(1)}-X_{t}^{(2)}%
$, $\hat{Y}_{t}=Y_{t}^{(1)}-Y_{t}^{(2)}$, $\hat{Z}_{t}=Z_{t}^{(1)}-Z_{t}%
^{(2)}$. Since $(X^{(i)},Y^{(i)},Z^{(i)},K^{(i)})$ satisfies $G$-FBSDE
(\ref{e-4-1}) for $i=1$, $2$ under $P^{\ast}$, we obtain $P^{\ast}$-a.s.%
\begin{equation}
\left \{
\begin{array}
[c]{rl}%
d\hat{X}_{t}= & \left[  a^{(1)}(t)\hat{X}_{t}+a^{(2)}(t)\hat{Y}_{t}\right]
dt+\left[  a^{(3)}(t)\hat{X}_{t}+a^{(4)}(t)\hat{Y}_{t}\right]  dB_{t},\\
d\hat{Y}_{t}= & \left[  \langle a^{(5)}(t),\hat{X}_{t}\rangle+a^{(6)}%
(t)\hat{Y}_{t}+a^{(7)}(t)\hat{Z}_{t}\right]  dt+\hat{Z}_{t}dB_{t}+dK_{t}%
^{(1)},\\
\hat{X}_{0}= & 0\in \mathbb{R}^{n},\text{ }\hat{Y}_{T}=\langle a_{T}^{(8)}%
,\hat{X}_{T}\rangle+\phi_{1}(X_{T}^{(2)})-\phi_{2}(X_{T}^{(2)}),
\end{array}
\right.  \label{e-4-2}%
\end{equation}
where $a^{(1)}(t)=(a_{jk}^{(1)}(t))_{j,k=1}^{n}$ and $a^{(2)}(t)=(a_{1}%
^{(2)}(t),\ldots,a_{n}^{(2)}(t))^{T}$ with%
\[
a_{jk}^{(1)}(t)=\left[  b_{j}(t,k-1)-b_{j}(t,k)+\left(  h_{j}(t,k-1)-h_{j}%
(t,k)\right)  \gamma_{t}\right]  (\hat{X}_{k,t})^{-1}I_{\{ \hat{X}_{k,t}%
\not =0\}},
\]%
\[
a_{j}^{(2)}(t)=\left[  b_{j}(t,X_{t}^{(2)},Y_{t}^{(1)})-b_{j}(t,X_{t}%
^{(2)},Y_{t}^{(2)})+\left(  h_{j}(t,X_{t}^{(2)},Y_{t}^{(1)})-h_{j}%
(t,X_{t}^{(2)},Y_{t}^{(2)})\right)  \gamma_{t}\right]  (\hat{Y}_{t})^{-1}I_{\{
\hat{Y}_{t}\not =0\}},
\]%
\[
b_{j}(t,k)=b_{j}(t,X_{1,t}^{(2)},\ldots,X_{k,t}^{(2)},X_{k+1,t}^{(1)}%
,\ldots,X_{n,t}^{(1)},Y_{t}^{(1)}),
\]
similar for the definition of notations $b_{j}(t,k-1)$, $h_{j}(t,k-1)$,
$h_{j}(t,k)$, $a^{(3)}(t)$, $a^{(4)}(t)$, $a^{(5)}(t)$, $a^{(6)}(t)$,
$a^{(7)}(t)$ and $a_{T}^{(8)}$. From the assumption (H2), it is easy to verify
that%
\[
|a^{(1)}(t)|\leq nL_{1}(1+\bar{\sigma}^{2})\text{, }|a^{(2)}(t)|\leq
nL_{2}(1+\bar{\sigma}^{2})\text{, }|a^{(3)}(t)|\leq nL_{1}\text{, }%
|a^{(4)}(t)|\leq nL_{2}\text{,}%
\]%
\[
|a^{(5)}(t)|\leq L_{3}(1+\bar{\sigma}^{2})\text{, }|a^{(6)}(t)|+|a^{(7)}%
(t)|\leq L_{1}(1+\bar{\sigma}^{2})\text{, }|a_{T}^{(8)}|\leq L_{3}\text{.}%
\]

Consider the following FBSDE under $P^{\ast}$:%
\begin{equation}
\left \{
\begin{array}
[c]{rl}%
dl_{t}= & \left[  -a^{(6)}(t)l_{t}+\langle a^{(2)}(t),p_{t}\rangle
+\langle \gamma_{t}a^{(4)}(t),q_{t}\rangle \right]  dt-\gamma_{t}^{-1}%
a^{(7)}(t)l_{t}dB_{t},\\
dp_{t}= & \left[  l_{t}a^{(5)}(t)-a^{(1)}(t)p_{t}-\gamma_{t}a^{(3)}%
(t)q_{t}\right]  dt+q_{t}dB_{t}+dN_{t},\\
l_{0}= & 1,\text{ }p_{T}=l_{T}a_{T}^{(8)}\in \mathbb{R}^{n},
\end{array}
\right.  \label{e-4-3}%
\end{equation}
where $N$ is a $\mathbb{R}^{n}$-valued square integrable martingale with
$N_{0}=0$ such that each component of $N$ is orthogonal to $B$ under $P^{\ast
}$. By Theorem 6.1 in \cite{EH}, for each $(l_{t})_{t\leq T}\in S_{P^{\ast}%
}^{2}(0,T)$, the BSDE
\[
dp_{t}=\left[  l_{t}a^{(5)}(t)-a^{(1)}(t)p_{t}-\gamma_{t}a^{(3)}%
(t)q_{t}\right]  dt+q_{t}dB_{t}+dN_{t},\text{ }p_{T}=l_{T}a_{T}^{(8)}\text{,}%
\]
has a unique $L^{2}$-solution $(p,q,N)$ with $p\in S_{P^{\ast}}^{2}%
(0,T;\mathbb{R}^{n})$ and $q\in M_{P^{\ast}}^{2,2}(0,T;\mathbb{R}^{n})$, where
$S_{P^{\ast}}^{2}(0,T)$ (resp. $M_{P^{\ast}}^{2,2}(0,T)$) is the completion of
$S^{0}(0,T)$ (resp. $M^{0}(0,T)$) under the norm
\[
||\eta||_{S_{P^{\ast}}^{2}(0,T)}:=\left(  E_{P^{\ast}}\left[  \sup_{t\leq
T}|\eta_{t}|^{2}\right]  \right)  ^{1/2}\text{ }\left(  \text{resp. }%
||\eta||_{M_{P^{\ast}}^{2,2}(0,T)}:=\left(  E_{P^{\ast}}\left[  \int_{0}%
^{T}|\eta_{t}|^{2}dt\right]  \right)  ^{1/2}\right)  \text{.}%
\]
Similar to the proof of Theorem \ref{th-3-1}, we can deduce that there exists
a $\delta_{2}>0$ depending on $n$, $T$, $L_{1}$, $\bar{\sigma}$ and
$\underline{\sigma}$ such that FBSDE (\ref{e-4-3}) has a unique $L^{2}%
$-solution $(l,p,q,N)$ under the assumption $L_{2}L_{3}<\delta_{2}$.

Taking $\delta=\delta_{1}\wedge \delta_{2}$, we assume $L_{2}L_{3}<\delta$ in
the following. Applying It\^{o}'s formula to $\langle p_{t},\hat{X}_{t}%
\rangle-l_{t}\hat{Y}_{t}$ under $P^{\ast}$, we obtain%
\begin{equation}
\hat{Y}_{0}=E_{P^{\ast}}\left[  l_{T}\left(  \phi_{1}(X_{T}^{(2)})-\phi
_{2}(X_{T}^{(2)})\right)  -\int_{0}^{T}l_{t}dK_{t}^{(1)}\right]  .
\label{e-4-4}%
\end{equation}
Since $\phi_{1}(X_{T}^{(2)})\geq \phi_{2}(X_{T}^{(2)})$ and $dK_{t}^{(1)}\leq
0$, we only need to prove $l_{t}\geq0$ $P^{\ast}$-a.s. for $t\in \lbrack0,T]$.
Define the stopping time%
\[
\tau=\inf \{t\geq0:l_{t}=0\} \wedge T.
\]
It is clear that $l_{\tau}=0$ on $\{ \tau<T\}$ and $l_{T}\geq0$ on $\{
\tau=T\}$. Consider the following FBSDE on $[\tau,T]$ under $P^{\ast}$:%
\begin{equation}
\left \{
\begin{array}
[c]{rl}%
dl_{t}^{\prime}= & \left[  -a^{(6)}(t)l_{t}^{\prime}+\langle a^{(2)}%
(t),p_{t}^{\prime}\rangle+\langle \gamma_{t}a^{(4)}(t),q_{t}^{\prime}%
\rangle \right]  dt-\gamma_{t}^{-1}a^{(7)}(t)l_{t}^{\prime}dB_{t},\\
dp_{t}^{\prime}= & \left[  l_{t}^{\prime}a^{(5)}(t)-a^{(1)}(t)p_{t}^{\prime
}-\gamma_{t}a^{(3)}(t)q_{t}^{\prime}\right]  dt+q_{t}^{\prime}dB_{t}%
+dN_{t}^{\prime},\\
l_{\tau}^{\prime}= & l_{\tau},\text{ }p_{T}^{\prime}=l_{T}^{\prime}a_{T}%
^{(8)}\in \mathbb{R}^{n},\text{ }t\in \lbrack \tau,T].
\end{array}
\right.  \label{e-4-5}%
\end{equation}
It is easy to verify that%
\[
(l_{t}^{\prime},p_{t}^{\prime},q_{t}^{\prime},N_{t}^{\prime})_{t\in \lbrack
\tau,T]}=\left(  l_{T}I_{\{ \tau=T\}},l_{T}a_{T}^{(8)}I_{\{ \tau
=T\}},0,0\right)  _{t\in \lbrack \tau,T]}%
\]
satisfies FBSDE (\ref{e-4-5}). Obviously, $(l_{t}^{\prime},p_{t}^{\prime
},q_{t}^{\prime},N_{t}^{\prime})_{t\in \lbrack \tau,T]}=(l_{t},p_{t},q_{t}%
,N_{t}-N_{\tau})_{t\in \lbrack \tau,T]}$ satisfies FBSDE (\ref{e-4-5}). Since
the $L^{2}$-solution to FBSDE (\ref{e-4-5}) is unique, we obtain $l_{t}%
=l_{T}I_{\{ \tau=T\}}$ for $t\in \lbrack \tau,T]$. Thus $l_{t}\geq0$ $P^{\ast}%
$-a.s. for $t\in \lbrack0,T]$. By (\ref{e-4-4}), we get $\hat{Y}_{0}\geq0$,
which implies (ii).
\end{proof}

Suppose $n=1$ in the following and consider the following $G$-FBSDEs:%
\begin{equation}
\left \{
\begin{array}
[c]{rl}%
dX_{t}^{(i)}= & b(t,X_{t}^{(i)},Y_{t}^{(i)})dt+h(t,X_{t}^{(i)},Y_{t}%
^{(i)})d\langle B\rangle_{t}+\sigma(t,X_{t}^{(i)},Y_{t}^{(i)})dB_{t},\\
dY_{t}^{(i)}= & f(t,X_{t}^{(i)},Y_{t}^{(i)},Z_{t}^{(i)})dt+g(t,X_{t}%
^{(i)},Y_{t}^{(i)},Z_{t}^{(i)})d\langle B\rangle_{t}+Z_{t}^{(i)}dB_{t}%
+dK_{t}^{(i)},\\
X_{0}^{(i)}= & x_{i}\in \mathbb{R},\text{ }Y_{T}^{(i)}=\phi(X_{T}^{(i)}),\text{
}i=1,2.
\end{array}
\right.  \label{e-4-6}%
\end{equation}

\begin{theorem}
\label{th-4-2}Suppose that assumptions (H1) and (H2) hold with $n=1$ and
$\beta>2$. Then there exists a $\delta>0$ depending on $T$, $L_{1}$,
$\bar{\sigma}$ and $\underline{\sigma}$ such that the following results hold.

\begin{description}
\item[(i)] If $L_{2}L_{3}<\delta$, then $G$-FBSDE (\ref{e-4-6}) has a unique
$L^{2}$-solution $(X^{(i)},Y^{(i)},Z^{(i)},K^{(i)})$ for $i=1$, $2$.

\item[(ii)] If $L_{2}L_{3}<\delta$, $x_{1}\geq x_{2}$, $\phi(\cdot)$ is
non-decreasing, $f(\cdot)$ and $g(\cdot)$ are non-increasing in $x$, then we
have $Y_{0}^{(1)}\geq Y_{0}^{(2)}$.
\end{description}
\end{theorem}

\begin{proof}
The proof is similar to the proof of Theorem \ref{th-4-1}. For the convenience
of the reader, we sketch the proof. (i) is obvious. For (ii), we can similarly
find a $P^{\ast}\in \mathcal{P}$ such that $K_{T}^{(2)}=0$ $P^{\ast}$-a.s. The
equation (\ref{e-4-2}) is rewritten as the following equation: $P^{\ast}$-a.s.%
\begin{equation}
\left \{
\begin{array}
[c]{rl}%
d\hat{X}_{t}= & \left[  a^{(1)}(t)\hat{X}_{t}+a^{(2)}(t)\hat{Y}_{t}\right]
dt+\left[  a^{(3)}(t)\hat{X}_{t}+a^{(4)}(t)\hat{Y}_{t}\right]  dB_{t},\\
d\hat{Y}_{t}= & \left[  a^{(5)}(t)\hat{X}_{t}+a^{(6)}(t)\hat{Y}_{t}%
+a^{(7)}(t)\hat{Z}_{t}\right]  dt+\hat{Z}_{t}dB_{t}+dK_{t}^{(1)},\\
\hat{X}_{0}= & x_{1}-x_{2},\text{ }\hat{Y}_{T}=a_{T}^{(8)}\hat{X}_{T},
\end{array}
\right.  \label{e-4-7}%
\end{equation}
where the notations $a^{(1)}(t)$, $a^{(2)}(t)$, $a^{(3)}(t)$, $a^{(4)}(t)$,
$a^{(5)}(t)$, $a^{(6)}(t)$ and $a^{(7)}(t)$ are the same as the notations in
the proof of Theorem \ref{th-4-1} under $n=1$,%
\[
a_{T}^{(8)}=\left[  \phi(X_{T}^{(1)})-\phi(X_{T}^{(1)})\right]  (\hat{X}%
_{T})^{-1}I_{\{ \hat{X}_{T}\not =0\}}.
\]
Since $\phi(\cdot)$ is non-decreasing, $f(\cdot)$ and $g(\cdot)$ are
non-increasing in $x$, it is easy to verify that
\begin{equation}
a_{T}^{(8)}\geq0\text{ and }a^{(5)}(t)\leq0\text{ for }t\in \lbrack0,T].
\label{e-4-8}%
\end{equation}
Applying It\^{o}'s formula to $p_{t}\hat{X}_{t}-l_{t}\hat{Y}_{t}$ under
$P^{\ast}$, where $(l,p,q,N)$ is the $L^{2}$-solution of FBSDE (\ref{e-4-3})
under $n=1$, we obtain%
\[
\hat{Y}_{0}=p_{0}(x_{1}-x_{2})+E_{P^{\ast}}\left[  -\int_{0}^{T}l_{t}%
dK_{t}^{(1)}\right]  .
\]
We have obtained $l_{t}\geq0$ $P^{\ast}$-a.s. for $t\in \lbrack0,T]$ in the
proof of Theorem \ref{th-4-1}. Thus we get%
\begin{equation}
\hat{Y}_{0}\geq p_{0}(x_{1}-x_{2}). \label{e-4-9}%
\end{equation}
By (\ref{e-4-8}), we have%
\[
l_{T}a_{T}^{(8)}\geq0\text{ and }l_{t}a^{(5)}(t)\leq0\text{ for }t\in
\lbrack0,T].
\]
By comparison theorem for BSDEs
\[
dp_{t}=\left[  l_{t}a^{(5)}(t)-a^{(1)}(t)p_{t}-\gamma_{t}a^{(3)}%
(t)q_{t}\right]  dt+q_{t}dB_{t}+dN_{t},\text{ }p_{T}=l_{T}a_{T}^{(8)},
\]
and%
\[
d\tilde{p}_{t}=\left[  -a^{(1)}(t)\tilde{p}_{t}-\gamma_{t}a^{(3)}(t)\tilde
{q}_{t}\right]  dt+\tilde{q}_{t}dB_{t}+d\tilde{N}_{t},\text{ }\tilde{p}%
_{T}=0,
\]
we get $p_{0}\geq \tilde{p}_{0}=0$. Thus, from (\ref{e-4-9}), we deduce
$\hat{Y}_{0}\geq0$, which implies (ii).
\end{proof}


\bigskip

\begin{thebibliography}{99}                                                                                               %


\bibitem {An}F. Antonelli, Backward-forward stochastic differential equations,
Ann. Appl. Probab., 3 (1993), 777-793.

\bibitem {ALP}M. Avellaneda, A. Levy, A. Paras, Pricing and hedging derivative
securities in markets with uncertain volatilities, Appl. Math. Finance, 2
(1995), 73-88.

\bibitem {CST}P. Cheridito, H. Soner, N. Touzi, N. Victoir, Second order
backward stochastic differential equations and fully nonlinear parabolic pdes,
Commun. Pure Appl. Math., 60(7) (2007), 1081-1110.

\bibitem {De}F. Delarue, On the existence and uniqueness of solutions to
FBSDEs in a non-degenerate case, Stoch. Process. Appl., 99 (2002), 209-286.

\bibitem {DHP11}L. Denis, M. Hu, S. Peng, Function spaces and capacity related
to a sublinear expectation: application to $G$-Brownian motion paths,
Potential Anal., 34 (2011), 139-161.

\bibitem {EH}N. El Karoui, S. Huang, A general result of existence and
uniqueness of backward stochastic differential equations, in Backward
Stochastic Differential Equations, N. El Karoui, and L. Mazliak, eds., Pitman
Res. Notes Math. Ser., 364, Longman, Harlow, 1997, 27-36.

\bibitem {HJPS1}M. Hu, S. Ji, S. Peng, Y. Song, Backward stochastic
differential equations driven by G-Brownian motion, Stochastic Process. Appl.,
124 (2014), 759-784.

\bibitem {HJPS}M. Hu, S. Ji, S. Peng, Y. Song, Comparison theorem, Feynman-Kac
formula and Girsanov transformation for BSDEs driven by G-Brownian motion,
Stochastic Process. Appl., 124 (2014), 1170-1195.

\bibitem {HJX}M. Hu, S. Ji, X. Xue, A global stochastic maximum principle for
fully coupled forward-backward stochastic systems, SIAM J. Control, Optim.,
56(6) (2018), 4309-4335.

\bibitem {HJX1}M. Hu, S. Ji, X. Xue, The existence and uniqueness of viscosity
solution to a kind of Hamilton-Jacobi-Bellman equation, SIAM J. Control
Optim., 57 (2019), 3911-3938.

\bibitem {HP09}M. Hu, S. Peng, On representation theorem of G-expectations and
paths of $G$-Brownian motion, Acta Math. Appl. Sin. Engl. Ser., 25 (2009), 539-546.

\bibitem {HYH}Y. Hu, Y. Lin, A. S. Hima, Quadratic backward stochastic
differential equations driven by G-Brownian motion: Discrete solutions and
approximation, Stochastic Process. Appl., 128 (2018), 3724-3750.

\bibitem {HuP}Y. Hu, S Peng, Solution of forward-backward stochastic
differential equations, Probab. Theory Related Fields, 103(2) (1995), 273-283.

\bibitem {LRT}Y. Lin, Z. Ren, N. Touzi, J. Yang, Second order backward SDE
with random terminal time, Electron. J. Probab., 25 (2020), 1-43.

\bibitem {Liu}G. Liu, Multi-dimensional BSDEs driven by G-Brownian motion and
related system of fully nonlinear PDEs, Stoch. Int. J. Probab. Stoch.
Process., 92(6) (2019), 1-25.

\bibitem {LS}H. Lu, Y. Song, Forward-backward stochastic differential
equations driven by G-Brownian motion, arXiv:2104.06868v1, (2021).

\bibitem {Ly}T. Lyons, Uncertain volatility and the risk-free synthesis of
derivatives, Appl. Math. Finance, 2 (1995), 117-133.

\bibitem {MPY}J. Ma, P. Protter, J. Yong, Solving forward-backward stochastic
differential equations explicitly-a four step scheme, Probab. Theory Related
Fields, 98(2) (1994), 339-359.

\bibitem {MWZZ}J.Ma, Z.Wu, D. Zhang, J. Zhang, On well-posedness of
forward-backward SDEs-A unified approach, Ann. Appl. Probab., 25(4) (2015), 2168-2214.

\bibitem {MY}J. Ma, J. Yong, Forward-backward stochastic differential
equations and their applications, Springer Science \& Business Media, (1999).

\bibitem {P-T}E. Pardoux, S. Tang, Forward-backward stochastic differential
equations and quasilinear parabolic PDEs, Probab. Theory Related Fields,
114(2) (1999), 123-150.

\bibitem {P07a}S. Peng, $G$-expectation, $G$-Brownian Motion and Related
Stochastic Calculus of It\^{o} type, Stochastic analysis and applications,
Abel Symp., Vol. 2, Springer, Berlin, 2007, 541-567.

\bibitem {P08a}S. Peng, Multi-dimensional $G$-Brownian motion and related
stochastic calculus under $G$-expectation, Stochastic Process. Appl., 118
(2008), 2223-2253.

\bibitem {P2019}S. Peng, Nonlinear Expectations and Stochastic Calculus under
Uncertainty, Springer (2019).

\bibitem {PW}S. Peng, Z. Wu, Fully coupled forward-backward stochastic
differential equation and applications to optimal control, SIAM J. Control
Optim., 37(3) (1999), 825-843.

\bibitem {STZ}H. M. Soner, N. Touzi, J. Zhang, Martingale Representation
Theorem under G-expectation, Stochastic Process. Appl., 121 (2011), 265-287.

\bibitem {STZ11}H. M. Soner, N. Touzi, J. Zhang, Wellposedness of Second Order
Backward SDEs, Probab. Theory Related Fields, 153 (2012), 149-190.

\bibitem {Song11}Y. Song, Some properties on G-evaluation and its applications
to G-martingale decomposition, Sci. China Math., 54(2) (2011), 287-300.

\bibitem {WY}B. Wang, M. Yuan, Forward-backward stochastic differential
equations driven by G-Brownian motion, Appl. Math. Comput., 349 (2019), 39-47.

\bibitem {Wu}Z. Wu, The comparison theorem of FBSDE, Stat. Probab. Lett.,
44(1) (1999), 1-6.

\bibitem {Yong}J. Yong, Finding adapted solution of forward-backward
stochastic differential equations-method of continuation, Probab. Theory
Related Fields, 107(4) (1997), 537-572.

\bibitem {Zh}G. Zheng, Local wellposedness of coupled backward stochastic
differential equations driven by G-Brownian motions, J. Math. Anal. Appl.,
506(1) (2022), 1-18.
\end{thebibliography}
\end{document}